\documentclass{amsart}
\usepackage[margin=1in]{geometry}
\usepackage{amsmath}
\usepackage{amssymb}
\usepackage{amsthm}
\usepackage{amscd}
\newcounter{pcounter}

\newcommand{\PP}{\Bbb P}

\newcommand{\ZZ}{\Bbb Z}
\newcommand{\RR}{\Bbb R}
\newcommand{\TT}{\Bbb T}
\newcommand{\NN}{\Bbb N}

\newcommand{\CC}{\Bbb C}
\newcommand{\ip}[1]{\langle #1 \rangle}

\newcommand{\widetidle}{\widetilde}

\newcommand{\varespilon}{\varepsilon}

\newcommand{\actson}{\curvearrowright}

\newtheorem{?}{Question}
\newtheorem{theorem}{Theorem}
\newtheorem{definition}[theorem]{Definition}
\newtheorem{proposition}[theorem]{Proposition}
\newtheorem{cor}[theorem]{Corollary}
\newtheorem{lemma}[theorem]{Lemma}

\newcommand{\FF}{\Bbb F}

\DeclareMathOperator{\BC}{BC}

\DeclareMathOperator{\vol}{vol}
\DeclareMathOperator{\vr}{vr}

\DeclareMathOperator{\Sym}{Sym}

\DeclareMathOperator{\Span}{Span}
\DeclareMathOperator{\ord}{ord}
\DeclareMathOperator{\Map}{Map}
\DeclareMathOperator{\im}{im}

\DeclareMathOperator{\rank}{Rank}
\DeclareMathOperator{\tr}{tr}
\DeclareMathOperator{\supp}{supp}

\DeclareMathOperator{\proj}{proj}

\DeclareMathOperator{\micr}{micr}
\DeclareMathOperator{\tmicr}{tmicr}
\DeclareMathOperator{\Hom}{Hom}

\DeclareMathOperator{\Tr}{Tr}

\DeclareMathOperator{\mdim}{mdim}

\DeclareMathOperator{\wdim}{wdim}
\DeclareMathOperator{\Ball}{Ball}

\numberwithin{theorem}{section}

\begin{document}
\title{ Metric Mean Dimension For Algebraic Actions of Sofic Groups}
\author{Ben Hayes}\thanks{University of California, Los Angeles
520 Portola Plaza
Los Angeles,CA 90095-155.\\ e-mail:brh6@ucla.edu. The author is grateful for support from NSF Grants DMS-1161411 and DMS-0900776.}

\date{\today}
\maketitle
\begin{abstract} We prove that if $\Gamma$ is a sofic group, and $A$ is a finitely generated $\ZZ(\Gamma)$-module, then the metric mean dimension of $\Gamma \actson \widehat{A},$ in the sense of Hanfeng Li is equal to the von Neumann-L\"{u}ck rank of $A.$ This partially extends  the results of Hanfeng Li and Bingbing Liang in \cite{LiLiang} from the case of amenable groups to the case of sofic groups. Additionally we show that the mean dimension of $\Gamma\actson \widehat{A}$ is the von Neumann-L\"{u}ck rank of $A,$ if $A$ is finitely presented and $\Gamma$ is residually finite. It turns out that our approach naturally leads to a notion of $p$-metric mean dimension, which is in between mean dimension and the usual metric mean dimension. This can be seen as an obstruction to the equality of mean dimension and metric mean dimension. While we cannot decide  if mean dimension is the same as metric mean dimension for algebraic actions, we show that  in the metric case that for all $p$ the $p$-metric mean dimension coincides with the von Neumann-L\"{u}ck rank of  the dual module. \end{abstract}

%

\tableofcontents
\section{Introduction}

Mean dimension was first introduced for amenable groups by Gromov in \cite{Gromov} and studied systematically by Lindenstrauss and Weiss in \cite{LindWeiss}. It  is a dynamical version of dimension, and can be thought of as topological entropy for actions on ``large" topological spaces.  Lindenstrauss and Weiss showed that if $\Gamma$ is amenable,  and if $\mdim(X,\Gamma)>0$ then the entropy of $\Gamma \actson X$ is infinite.  Additionally they showed that when $\Gamma$ is amenable and $Y$ is a compact topological manifold, then $\mdim(Y^{\Gamma},\Gamma)=\dim(Y).$ Further, Lindenstrauss and Weiss defined a metric version of mean dimension, called metric mean dimension. The metric mean dimension is defined by considering dynamical versions of packing numbers with respect to metrics on $X,$ and then taking the infimum over all such metrics. Recently, mean dimension and metric mean dimension have been extended to actions of  sofic groups by Hanfeng Li in \cite{Li}. Additionally, Li extended Lindenstrauss and Weiss' result on Bernoulli shifts to the context of all sofic groups.

	A particularly nice class of examples of actions of general groups are the \emph{algebraic} actions. Let $\Gamma$ be a countable discrete group, and let $A$ be a countable abelian group on which $\Gamma$ acts by automorphisms. Let $\widehat{A}$ be the group of  homomorphisms into $\TT=\RR/\ZZ.$ Then $\widehat{A}$ is a compact abelian group  when given the topology of pointwise convergence. We have an action $\Gamma\actson \widehat{A}$ given by
\[(g\chi)(a)=\chi(g^{-1}a),\]
 and  $\widehat{A}$ is a compact metrizable space, we call this an \emph{algebraic action} of $\Gamma.$ Further, the action clearly preserves the Haar measure on $\widehat{A}.$   Often one can describe the purely dynamical notions of $\Gamma\actson \widehat{A}$ (i.e. the ones that view $\widehat{A}$ as a compact metrizable space or measure space, and ignore the algebraic structure) in terms of the action of $\Gamma$ on $A.$ For example, one can completely describe ergodicity, mixing, and expansiveness for $\Gamma\actson \widehat{A},$ in terms of $\Gamma \actson A.$ Additionally, the numerical invariants of $\Gamma\actson \widehat{A},$ can often be described in terms of the $L^{2}$-invariants of $\Gamma \actson A,$ such as measure theoretic and topological entropy, and this has been extensively studied for amenable groups (see \cite{LiThom},\cite{LiLiang}, \cite{ChungThom},\cite{LindSchmidt1},\cite{LindSchmidt2},\cite{Schmidt},\cite{Den},\cite{DenSchmidt}). Part of the above work has been extended from the amenable case to the sofic case, including the computation of entropy of principal algebraic actions of residually finite groups   by Bowen (see Theorem 1.2 in \cite{BowenEntropy}) Kerr-Li (see Theorem 7.1 in \cite{KLi}), and Bowen-Li (see \cite{BowenLi} Theorem 1.3).

	As an example of the above phenomenon, Hanfeng Li and Bingbing Liang showed that if $\Gamma$ is amenable, then $\mdim(\widehat{A},\Gamma)=\vr(A),$ where $\vr(A)$ is the von Neumann--L\"{u}ck rank of $A$ as a $\ZZ(\Gamma)$-module (see Section \ref{S:prelim} for the definition). In this paper, we  show if $A$ is a finitely generated $\ZZ(\Gamma)$-module, that the \emph{metric} mean dimension of $\Gamma \actson \widehat{A}$ is equal to $\vr(A),$ when $\Gamma$ is \emph{sofic}. Additionally, we will show that when $A$ is a finitely presented $\ZZ(\Gamma)$-module, $\Gamma$ is residually finite, and the sofic approximation comes from a sequence of finite quotients, then the mean dimension of $\Gamma\actson \widehat{A}$ is equal to the von Neumann--L\"{u}ck rank of $A.$ This recovers the result of Li-Liang in the finitely presented amenable case.

	There is a fair amount of difficulty in showing the mean dimension of $\Gamma \actson \widehat{A}$ is $\vr(A),$ which will be explained later in the paper. Our techniques do not entirely recover the results of Li-Liang for amenable groups, and it is not clear  how one would fix this. However, we now know that $\vr(A)$ is an invariant of $\Gamma \actson \widehat{A}$ as a \emph{topological dynamical system} (ignoring the algebraic structure of $\widehat{A}$) for sofic groups, whereas the results of Li-Liang only show that for amenable groups. Finally, we shall present $\ell^{p}$-versions of metric mean dimension, and show that they are all bounded below by the usual mean dimension. There is an obvious chain of inequalities between these mean dimensions.  As discussed in the section on questions and conjectures, if one can prove that $p$-metric mean dimension and $q$-metric mean dimension are different for some values of $p,q$ this would prove that metric mean dimension is not the same as mean dimension (which is unknown in general). Additionally, if one can prove that $p$-metric mean dimension always equals $q$-metric mean dimension, then this gives added flexibility in dealing with metric mean dimension. Thus, it is an interesting question to decide if these different versions of metric mean dimension are equal or not, independent of the outcome.  Additionally, we shall give an alternate formula for amenable groups which shows that we can compute $p$-metric mean dimension in a manner similar to the usual formula for entropy for amenable groups.

	Some of the techniques for this paper are already inherent in \cite{Me},\cite{Me2} and \cite{Li}. In particular, the idea of considering microstates on a quotient to just be microstates on the larger space which are small on the kernel is crucial in \cite{Me},\cite{Me2}, and is an important conceptual step in our proof of the main theorem. Similarly, we will use that $\vol(\Ball(l^{p}(n,\mu_{n})))^{1/n}$  (here $\mu_{n}$ is the uniform probability measure) is bounded to prove that $p$-metric mean dimension is bounded below by mean dimension.  Consequences of the control on $\vol(\Ball(l^{p}(n,\mu_{n}))^{1/n}$ are already present in the proof of Theorem 4.7 in \cite{Me}, as well as Theorem 6.2 in \cite{Me2}. Lastly the notion of ``microstates rank" is an alternate version of the entropic formula for von Neumann dimension developed in Theorem 6.6 in \cite{Me}, and Theorem 6.4 in \cite{Me2}. Additionally, certain portions of the paper are adapting the techniques in \cite{Li} (some of which can be traced back to \cite{LindWeiss}) particularly in subsection \ref{SS:meandimension}. Structurally,  the paper follows \cite{LiLiang}, in particular the proof of the main theorem follows a similar format to \cite{LiLiang} and this is intentional. However, as is typical in passing from the amenable entropy to sofic entropy, the techniques in this paper are often quite different from those in \cite{LiLiang}. For example, a crucial step in the proof of the main theorem in \cite{LiLiang} is proving that
\[\vr(A)=\sup_{E\subseteq A\mbox{ finite }}\lim_{n\to \infty}\frac{\rank{\ip{F_{n}E}}}{|F_{n}|},\]
where $F_{n}$ is a F\o lner sequence,  $\ip{F_{n}E}$ is the group generated by $F_{n}E,$ and $\rank(\cdot)$ is the usual rank of a finitely generated abelian group. Taking limits of precise algebra will not work in our context, and instead we consider exponential growth rates of $\varepsilon$-packing numbers of certain subsets of $\Hom(A,\ell^{2}(n)).$ As $\dim_{\CC}(\Hom(A,\ell^{2}(n)))=n\rank(A),$ this can be seen as replacing rank in the above formula with an $\varepsilon$-rank. This appears to be necessary in the nonamenable context, and is an important philosophy behind the techniques in \cite{Me},\cite{Me5}. Lastly, it should be mentioned that the recent work on sofic entropy and related areas originates in \cite{Bow} and \cite{KLi}.

\textbf{Acknowledgments} I would like to thank Dimitri Shlyakhtenko for his continuing support and helpful suggestions on this problem. I would like to thank Hanfeng Li and Lewis Bowen for comments on an earlier version of this paper, which improved the paper greatly. I would additionally like to thank the anonymous referee for their numerous comments, which greatly improved the paper. Much of this paper developed from discussions at the  Von Neumann Algebras and Measurable Group Theory conference in July 2014 in Leuven, the Dynamics, Geometry and Operator Algebras conference in Texas A\&M in August 2014,  and the  $C^{*}$-Algebren conference in Oberwolfach in  August 2014.

\subsection{Notation and Terminology}

	The identity element of a group $\Gamma$ will be denoted $e,$ unless $\Gamma$ is assumed to be abelian in which case we will use $0$ for the identity element. If $A$ is an abelian group, we will write its group operation additively unless otherwise stated. In particular, we will use $\TT$ for the additive group $\RR/\ZZ$ and not the multiplicative group of modulus one complex numbers, and similarly for $\RR^{n}/\ZZ^{n}.$

	We will assume some familiarity with measurable functional calculus for normal elements in $B(\mathcal{H}).$ If $T\in B(\mathcal{H})$ is normal, and $f$ is a  bounded real-valued Borel function on the spectrum of $T,$ we will use $f(T)$ for the operator obtained by integrating $f$ against the spectral measure of $T.$ For the reader's convenience we note the most often used example: if $\mathcal{H}$ is finite dimensional and $A\subseteq \CC$ is Borel, then $\chi_{A}(T)$ is the orthogonal projection onto the eigenspaces whose corresponding eigenvalues are in $A.$ Throughout the paper we will use notation analogous to measure theory for tracial von Neumann algebras (defined in section \ref{S:prelim}). For example if $(M,\tau)$ is a tracial von Neumann algebra, the operator norm of an element $x\in M,$ will often be denoted $\|x\|_{\infty},$ and we  will use $\|x\|_{p}^{p}=\tau(|x|^{p}),$ where $|x|=(x^{*}x)^{1/2}.$

	A pseudometric $d$ on a set $X,$ is a function $d\colon X\times X\to [0,\infty)$ satisfying symmetry and the triangle inequality, however we might have $d(x,y)=0$ and $x\ne y.$ A set $X$ with a pseudometric $d$ is a \emph{pseudometric space}. If $\|\cdot\|$ is a pseudonorm on $\RR^{n},$ we will use
\[\|a+\ZZ^{n}\|=\inf_{k\in \ZZ^{n}}\|a+k\|,\]
this induces a pseudometric $\rho$ on $\TT^{n}$ by
\[\rho(a,b)=\|a-b\|.\]
If $(X,d)$ is a pseudometric space, and $A,B\subseteq X,$ we write $A\subseteq_{\varepsilon}B$ to mean that $A$ is contained in the $\varespilon$ neighborhood of $B.$ A set $A\subseteq X$ is said to be $\varepsilon$-dense if $X\subseteq_{\varepsilon}A.$   We let $S_{\varepsilon}(X,d)$ be the smallest cardinality of an $\varepsilon$-dense subset. Note that if $A\subseteq_{\delta}B,$ then
\begin{equation}\label{E:canonicalperturbationinequality}
S_{2(\varepsilon+\delta)}(A,d)\leq S_{\varepsilon}(B,d).
\end{equation}
We use $\mu_{n}$ for the uniform probability measure on $\{1,\cdots,n\}.$ If $A$ is a finite set, we use $\mu_{A}$ for the uniform probability measure on $A.$ Lastly, we use $\tr$ for $\frac{1}{n}\Tr$ on $M_{n}(\CC).$ Throughout the paper we will use the term ``microstate" for any (sufficiently nice) finitary model of our structure. This will usually be almost equivariant maps into some finite set or finite-dimensional space. Thus the term microstate is not well-defined. This will not cause us any problem as we will only use it in heuristic terms.

\section{Preliminaries and the von Neumann Dimension Lemma}\label{S:prelim}

\subsection{ Tracial von Neumann Algebras and von Neumann-L\"{u}ck Rank}

\begin{definition}\emph{ Let $\mathcal{H}$ be a Hilbert space. A unital subalgebra $M\subseteq B(\mathcal{H})$ is a} von Neumann algebra \emph{ if it is closed under adjoints and in the weak operator topology. A} faithful normal tracial state \emph{on $M$ is a linear functional $\tau\colon M\to \CC$ satisfying the following hypotheses:}\end{definition}

\begin{list}{ \arabic{pcounter}:~}{\usecounter{pcounter}}
\item $\tau(1)=1,$\\
\item $\tau(x^{*}x)\geq 0,$  with equality if and only if $x=0$,\\
\item $\tau(xy)=\tau(yx),$  for all $x,y\in M,$
\item $\tau\big|_{\{x\in M:\|x\|\leq 1\}}$ is weak operator topology continuous.
\end{list}

In the above, $\|\cdot\|$ is the operator norm. The pair $(M,\tau)$ is called a \emph{tracial} von Neumann algebra.

	Here is the main example of a tracial von Neumann of concern for this paper. Let $\Gamma$ be a countable discrete group, and $\lambda\colon \Gamma\to \mathcal{U}(\ell^{2}(\Gamma))$ its left regular representation defined by $\lambda(g)f(x)=f(g^{-1}x).$ The group von Neumann algebra  $L(\Gamma)$ is defined by $L(\Gamma)=\overline{\Span}^{\mbox{WOT}}\{\lambda(g):g\in \Gamma\},$ where WOT is the weak operator topology. For $x\in L(\Gamma),$ set $\tau(x)=\ip{x\delta_{e},\delta_{e}}$. It is known that $\tau$ is a faithful normal tracial state on $L(\Gamma).$ We will call $\tau$ the \emph{group trace}. We identify $\CC(\Gamma)$ as a $*$-subalgebra of $L(\Gamma).$ We remark that the $*$-structure on $\CC(\Gamma)$ is given by
\[\left(\sum_{g\in\Gamma}\alpha_{g}g\right)^{*}=\sum_{g\in\Gamma}\overline{\alpha_{g^{-1}}}g.\]
 Under this identification, for $f\in\CC(\Gamma)$ we set $\widehat{f}(g)=\tau(f \lambda(g^{-1})).$ Thus,
\[f=\sum_{g\in \Gamma}\widehat{f}(g)\lambda(g).\]

	We also need to talk about the \emph{right} group von Neumann algebra. For this, consider the right regular representation
\[\rho\colon \Gamma\to \mathcal{U}(\ell^{2}(\Gamma))\]
given by
\[(\rho(g)\xi)(x)=\xi(xg).\]
We let $R(\Gamma)=\overline{\Span}^{\mbox{WOT}}\{\rho(g):g\in\Gamma\}.$ We let
\[\tau(x)=\ip{x\delta_{e},\delta_{e}},\]
then $\tau$ is a trace on $R(\Gamma).$  Technically, we should use different letters for the trace on $R(\Gamma)$ and the one on $L(\Gamma).$ It will typically be  clear from context which von Neumann algebra we are talking about. If we feel the need to specify which algebra we are talking about, we will say $\tau_{R(\Gamma)}$ for the trace on $R(\Gamma)$ and $\tau_{L(\Gamma)}$ for the trace on $L(\Gamma).$ We need the important facts (see \cite{BO} Theorem 6.1.4)
\[\{T\in B(\ell^{2}(\Gamma)):[T,\lambda(g)]=0\mbox{ for all $g\in\Gamma$}\}=R(\Gamma),\]
\[\{T\in B(\ell^{2}(\Gamma)):[T,\rho(g)]=0\mbox{ for all $g\in\Gamma$}\}=L(\Gamma).\]

We extend $\rho$ to $\CC(\Gamma)$ in the usual way. We will also need to amplify these maps to the matricial levels. We define
\[\lambda\colon M_{m,n}(\CC(\Gamma))\to  B(\ell^{2}(\Gamma)^{\oplus n},\ell^{2}(\Gamma)^{\oplus m})\]
\[\rho\colon M_{m,n}(\CC(\Gamma))\to B(\ell^{2}(\Gamma)^{\oplus n},\ell^{2}(\Gamma)^{\oplus m})\]
by
\[(\lambda(f)\xi)(j)=\sum_{l=1}^{n}f_{jl}\xi_{l}\]
\[(\rho(f)\xi)(j)=\sum_{l=1}^{n}\rho(f_{jl})\xi_{l}.\]
Recall that we are identifying $\CC(\Gamma)\subseteq L(\Gamma).$ We note here that if you are to view $\ell^{2}(\Gamma)^{\oplus n}$ as column vectors, then $\rho(f)$ is not multiplication by $f$! The reader should simply think of $\rho$ as \emph{some} representation of $\Gamma,$ and we are just applying  general procedure to induce this representation to an ``action'' of $M_{m,n}(\CC(\Gamma))$ (by ``action'' we mean $\rho(AB)=\rho(A)\rho(B)$ for $A\in M_{m,n}(\CC(\Gamma)),B\in M_{n,k}(\CC(\Gamma))$).

We will often need to switch from left to right multiplication in several proofs. For this, it is helpful to introduce the following partial adjoint operation. Given $x\in L(\Gamma)^{\oplus n}$ with
\[x=\begin{bmatrix}
x_{1}\\
x_{2}\\
\vdots\\
x_{n}
\end{bmatrix},\]
define $\widetilde{x}\in M_{1,n}(L(\Gamma))$ by
\[\widetilde{x}=\begin{bmatrix}
x_{1}& x_{2}& x_{3}& \cdots& x_{n}
\end{bmatrix}.\]
Since we view $\CC(\Gamma)\subseteq L(\Gamma),$ this operation also applies to elements of $\CC(\Gamma)^{\oplus n}.$ We call this a partial adjoint since it does not coincide with full adjoint on $L(\Gamma)^{\oplus n}\cong M_{n,1}(L(\Gamma))$ given by
\[\begin{bmatrix}
x_{1}\\
x_{2}\\
\vdots\\
x_{n}
\end{bmatrix}^{*}=\begin{bmatrix}
x_{1}^{*}&x_{2}^{*}&\cdots&x_{n}^{*}
\end{bmatrix}.\]

\begin{definition}\emph{Let $(M,\tau)$ be a tracial von Neumann algebra. Define $\tau\otimes \Tr\colon M_{n}(M)\to \CC$ by}
\[\tau\otimes \Tr(A)=\sum_{i=1}^{n}\tau(A_{ii}).\]
\end{definition}

	If $A$ is  a finitely generated projective module over $M,$ then $A\cong M^{\oplus n}p,$ for some $n\in\NN$ and some idempotent $p\in M_{n}(M).$ It thus makes sense to define $\dim_{M}(A)=\tau\otimes Tr(p),$ and this does not depend on $p.$ This essentially goes back to Murray and von Neumann. L\"{u}ck (see \cite{Luck1}) proved that one can extend this to arbitrary modules over $M,$ by defining
\[\dim_{M}(A)=\sup\{\dim_{M}(B):B\subseteq A \mbox{ is a finitely generated projective module}\}.\]

	This algebraicized von Neumann dimension. Additionally, many of the properties of von Neumann dimension persist in this extension, for example dimension is still additive under exact sequences. (For more, see \cite{Luck}, Chapter 6). Let $\Gamma$ be a countable discrete group, and $L(\Gamma)$ its von Neumann algebra equipped with the standard group trace. By the above, one can define the von Neumann-L\"{u}ck rank of a $\ZZ(\Gamma)$-module $A$ by
\[\vr(A)=\dim_{L(\Gamma)}(L(\Gamma)\otimes_{\ZZ(\Gamma)}A).\]
While L\"{u}ck algebraicized  von Neumann dimension, for us it will be more useful to go in the opposite direction. That is, we will need to use a more analytic definition of the von Neumann-L\"{u}ck rank for $\ZZ(\Gamma)$-modules. For this, let us recall the original definition due to Murray and von Neumann of von Neumann dimension. Suppose that $\mathcal{H}$ is a closed $\lambda(\Gamma)$-invariant subspace of $\ell^{2}(\Gamma)^{\oplus n}$ for some $n\in\NN.$ Then the projection onto $\mathcal{H},$ denoted $P_{\mathcal{H}},$ commutes with $\lambda(\Gamma)$, and thus by our preceding discussion we see that $P_{\mathcal{H}}\in M_{n}(R(\Gamma)).$ We can thus define
\[\dim_{R(\Gamma)}(\mathcal{H})=\Tr\otimes \tau(P_{\mathcal{H}}).\]
It is known that this agrees with viewing $\mathcal{H}$ as an algebraic left module over $L(\Gamma)$ and applying L\"{u}ck's definition of dimension. Similarly, if $\mathcal{H}$ is a closed $\rho(\Gamma)$-invariant subspace of $\ell^{2}(\Gamma)^{\oplus n}$, for some $n\in \NN,$ then $P_{\mathcal{H}}$ commutes with $\rho(\Gamma)$ and is thus in $M_{n}(L(\Gamma)).$ We then set
\[\dim_{L(\Gamma)}(\mathcal{H})=\Tr\otimes \tau(P_{\mathcal{H}}).\]
These are both special cases of von Neumann dimension as defined by Murray and von Neumann.

\begin{lemma}\label{L:hilred} Let $\Gamma$ be a countable discrete group, $n\in \NN,$ and $B\subseteq \ZZ(\Gamma)^{\oplus n}$ a left $\ZZ(\Gamma)$-submodule. Let $\eta\colon \CC(\Gamma)\to \ell^{2}(\Gamma)$ be the injection given by
\[\eta(f)(g)=\widehat{f}(g),\]
and let $\eta^{\oplus n}\colon \CC(\Gamma)^{\oplus n}\to \ell^{2}(\Gamma)^{\oplus n}$ be given by
\[\eta^{\oplus n}\left(\begin{bmatrix}
\alpha_{1}\\
\alpha_{2}\\
\vdots\\
\alpha_{n}\\
\end{bmatrix}\right)=
\begin{bmatrix}
\eta(\alpha_{1})\\
\eta(\alpha_{2})\\
\vdots\\
\eta(\alpha_{n})
\end{bmatrix}.\]
Let $\mathcal{H}_{B}=\overline{\Span}^{\|\cdot\|_{2}}\{\lambda \eta^{\oplus n}(f):\lambda\in \CC,f\in B\}.$ Then
\[\vr(\ZZ(\Gamma)^{\oplus n}/B)=n-\dim_{R(\Gamma)}(\mathcal{H}_{B}).\]

\end{lemma}

\begin{proof} By right exactness of tensor products, we have the exact sequence
\[\begin{CD}  L(\Gamma)\otimes_{\ZZ(\Gamma)}B @>>> L(\Gamma)^{\oplus n} @>>> L(\Gamma)\otimes_{\ZZ(\Gamma)} (\ZZ(\Gamma)^{\oplus n}/B) @>>> 0.\end{CD}\]
Given $\xi\in L(\Gamma),$ define $\xi\otimes 1\in M_{n}(L(\Gamma))$ by
	\[(\xi\otimes 1)_{ij}=\begin{cases}
	0,& \mbox{ if $i\ne j$}\\
	\xi,& \mbox{ if $i=j$}
	\end{cases}\]
Then
\begin{equation}\label{E:idaongao}\vr(\ZZ(\Gamma)^{\oplus n}/B)=\dim_{L(\Gamma)}( L(\Gamma)^{\oplus n}/L(\Gamma)B)=n-\dim_{L(\Gamma)}(L(\Gamma)B),
\end{equation}
where $L(\Gamma)B=\Span\{(\xi\otimes 1) f:\xi\in L(\Gamma),f\in B\}.$ Hence it suffices to show that $\dim_{L(\Gamma)}(L(\Gamma)B)=\dim_{R(\Gamma)}(\mathcal{H}_{B})$ for any submodule $B\subseteq \ZZ(\Gamma)^{\oplus n}.$ It is known that von Neumann dimension is continuous for taking inductive limits (see \cite{Luck} Theorem 6.3.1). So, it suffices to assume that $B$ is finitely generated as a $\ZZ(\Gamma)$-module, say by $f_{1},\cdots, f_{m}.$
 Let $f\in M_{m,n}(\ZZ(\Gamma))$ be given by
 \[f=\begin{bmatrix}
\widetilde{f_{1}}\\
\widetilde{f_{2}}\\
\vdots\\
\widetilde{f_{m}}
\end{bmatrix}.\]
A direct computation shows that for $\alpha\in \CC(\Gamma)^{\oplus m}$ with
\[\alpha=\begin{bmatrix}
\alpha_{1}\\
\alpha_{2}\\
\vdots\\
\alpha_{m}
\end{bmatrix}\] we have
\[\rho(f^{*})\alpha=\sum_{s=1}^{m}(\alpha_{s}\otimes 1)f_{s},\]
viewing $\CC(\Gamma)\subseteq L(\Gamma)$ and using the notation introduced at the beginning of the lemma. So
\[\mathcal{H}_{B}=\overline{\im(\rho(f^{*}))}.\]
By Lemma 5.4  in \cite{LiLiang}we know that
\[\vr(\ZZ(\Gamma)^{\oplus n}/B)=\dim_{L(\Gamma)}(\ker\lambda(f))\]
and so we have from (\ref{E:idaongao}) that
\[\dim_{L(\Gamma)}(L(\Gamma)B)=n-\vr(\ZZ(\Gamma)^{\oplus n}/B)=n-\dim_{L(\Gamma)}(\ker\lambda(f)).\]
Using that $\ell^{2}(\Gamma)^{\oplus n}=\ker(\lambda(f))\oplus \ker(\lambda(f))^{\perp}$ and Theorem 1.12 (2) of \cite{Luck} we have
\[n-\dim_{L(\Gamma)}(\ker\lambda(f))=\dim_{L(\Gamma)}(\ker\lambda(f)^{\perp})=\dim_{L(\Gamma)}(\overline{\im\lambda(f^{*})}).\]
So
\[\dim_{L(\Gamma)}(L(\Gamma)B)=\dim_{L(\Gamma)}(\overline{\im\lambda(f^{*})}).\]
Define
\[K\colon \ell^{2}(\Gamma)\to \ell^{2}(\Gamma)\]
by
\[(K\xi)(g)=\xi(g^{-1}),\]
and
\[U\colon \ell^{2}(\Gamma)^{\oplus m}\to \ell^{2}(\Gamma)^{\oplus m}\]
by
\[U\left(\begin{bmatrix}
\xi_{1}\\
\xi_{2}\\
\vdots\\
\xi_{m}
\end{bmatrix}\right)=\begin{bmatrix}
K(\xi_{1})\\
K(\xi_{2})\\
\vdots\\
K(\xi_{m})
\end{bmatrix}.\]
Then $U$ is a unitary and
\[UM_{n}(L(\Gamma))U^{*}=M_{n}(R(\Gamma)),\]
viewing $M_{n}(L(\Gamma)),M_{n}(R(\Gamma))$ as operators on $\ell^{2}(\Gamma)^{\oplus n}.$  Further
\[\Tr\otimes \tau_{R(\Gamma)}(UxU^{*})=\Tr\otimes \tau_{L(\Gamma)}(x),\]
for all $x\in M_{n}(L(\Gamma)).$  It follows that
\[\dim_{L(\Gamma)}(\overline{\im\lambda(f^{*})})=\dim_{R(\Gamma)}(\overline{\im\rho(f^{*})}).\]
Thus we conclude
\[\dim_{L(\Gamma)}(L(\Gamma)B)=\dim_{R(\Gamma)}(\overline{\im\rho(f^{*})})=\dim_{R(\Gamma)}(\mathcal{H}_{B}).\]

\end{proof}

We finally collect some technical lemmas on von Neumann dimension, which will be used frequently throughout the paper. For the next lemma we make some preliminary definitions.

\begin{definition}\emph{We say that a countable discrete group $\Gamma$ is} sofic, \emph{if there is a sequence of maps $\sigma_{i}\colon \Gamma\to S_{d_{i}}$  with $d_{i}\to \infty,$ such that}
\[\PP(\{1\leq k\leq d_{i}:\sigma_{i}(g)\sigma_{i}(h)(k)=\sigma_{i}(gh)(k)\})\to 1,\mbox{\emph{for all $g,h\in \Gamma$}}\]
\[\PP(\{1\leq k\leq d_{i}:\sigma_{i}(g)(k)\ne \sigma_{i}(h)(k)\})\to 1,\mbox{\emph{for all $g\ne h$ in $\Gamma$}}.\]
\emph{Here $\PP$ is with respect to the uniform probability measure on $\{1,\cdots,d_{i}\}.$ A sequence $\Sigma$ as above is called a} sofic approximation \emph{ of $\Gamma.$}\end{definition}

	It is known that all amenable groups and residually finite groups are sofic. Also, it is known that soficity is closed under free products (see \cite{ESZ1}), and under increasing unions. Additionally, residually sofic groups and locally sofic groups are sofic. This implies all linear groups are sofic. Finally, if $\Lambda\triangleleft \Gamma,$ with $\Lambda$ sofic, and $\Gamma/\Lambda$ amenable, then $\Gamma$ is sofic (see \cite{ESZ1}). See \cite{DKP} and \cite{KerrDykemaPichot2} for alternate proofs of these facts.

	Our main use of sofic groups (aside from metric mean dimension and mean dimension being defined) is that once we know how to model the group by permutations, we automatically get a way to model $\CC(\Gamma)$ and $M_{n,m}(\CC(\Gamma)$) by matrices. Indeed, we can extend $\sigma_{i}$ to a map
\[\sigma_{i}\colon \CC(\Gamma)\to M_{d_{i}}(\CC)\]
by
\[\sigma_{i}(f)=\sum_{g\in \Gamma}\widehat{f}(g)\sigma_{i}(g),\]
and further to a map
\[\sigma_{i}\colon M_{n,m}(\CC(\Gamma))\to M_{n,m}(M_{d_{i}}(\CC))\]
by
\[\sigma_{i}(A)_{jk}=\sigma_{i}(A_{jk}).\]

	It turns out that many spectral properties of elements in $M_{n,m}(\CC(\Gamma))$ are reflected by their approximates in $M_{n,m}(M_{d_{i}}(\CC))$ and this will be exploited several times throughout the paper. Moreover these approximations model $M_{n,m}(\RR(\Gamma))$ by real matrices, and similarly for $M_{n,m}(\ZZ(\Gamma)),$ and this will be important when we turn to mean dimension at the end of the paper. We now describe how spectral properties of elements in $M_{n,m}(\CC(\Gamma))$ are described by their approximates in $M_{n,m}(M_{d_{i}}(\CC)).$ For the next definition, we use $\|\cdot\|_{2}$ on $M_{n,m}(\CC)$ for the norm given by
\[\|A\|_{2}^{2}=\frac{1}{m}\Tr(A^{*}A).\]

\begin{definition}\emph{Let $(M,\tau)$ be a tracial von Neumann algebra. Let $f\in M_{n}(M)$ be a normal element, the} spectral measure for $x,$ \emph{(sometimes called the spectral measure of $x$ with respect to $\tau$) denoted $\mu_{f},$ is defined by $\mu_{f}(B)=\tau\otimes \Tr(\chi_{B}(f))$ if $B$ is a Borel subset of the spectrum of $x.$ It is an exercise to show that for all bounded complex valued Borel functions $g$ on the spectrum of $f,$ we have}
\[\tau\otimes \Tr(g(f))=\int g\,d\mu_{f}.\]
\end{definition}

The next Lemma is essentially folklore and we include the proof for completeness. Recall that if $(M,\tau)$ is a tracial von Neumann algebra, and $A,B\in M_{n,m}(M),$ then we define $\|\cdot\|_{2}$ on $M_{n,m}(M)$ by
\[\|A-B\|_{2}^{2}=\frac{1}{m}\tau\otimes \Tr(|A-B|^{2}).\]

\begin{lemma}\label{L:weak^{*}} Let $\Gamma$ be a countable discrete sofic group and $\sigma_{i}\colon \Gamma\to S_{d_{i}}$ a sofic approximation.  Extend $\sigma_{i}$ to $\CC(\Gamma)\to M_{d_{i}}(\CC)$ by $\sigma_{i}(f)=\sum_{g\in\Gamma}\widehat{f}(g)\sigma_{i}(g).$ Further extend  $\sigma_{i}$ to $M_{n,m}(\CC(\Gamma))\to M_{n,m}(M_{d_{i}}(\CC))$ by $\sigma_{i}(f)_{jk}=\sigma_{i}(f_{jk}).$    Fix $f\in M_{n,m}(\CC(\Gamma)),$ suppose that $A_{i}\in M_{n,m}(M_{d_{i}}(\CC))$ satisfy $\sup_{i}\|A_{i}\|_{\infty}<\infty,$ and  $\|A_{i}-\sigma_{i}(f)\|_{2}\to 0.$ Then $\mu_{|A_{i}|}\to \mu_{|f|}$ weak$^{*}.$

\end{lemma}

\begin{proof}  Let us first do this when $A_{i}=\sigma_{i}(f).$  Our claim is equivalent to the statement that
\[\tr\otimes \Tr(\phi(\sigma_{i}(f)^{*}\sigma_{i}(f)))\to \tau\otimes \Tr(\phi(f^{*}f))\]
for all continuous functions on $[0,\infty).$  It follows from the definition that
\[\tr\otimes \Tr(P(\sigma_{i}(f)^{*}\sigma_{i}(f)))\to \tau \otimes \Tr(P(f^{*}f)),\]
for all $P\in \CC[z].$ Since $\|\sigma_{i}(f)^{*}\sigma_{i}(f)\|_{\infty}$ is uniformly bounded,  the Weierstrass Approximation Theorem handles the case when $A_{i}=\sigma_{i}(f).$

	In the general case, again by the Weierstrass Approximation theorem we only have to show that
	\[\tr\otimes \Tr(|A_{i}|^{2k})\to \tau\otimes \Tr(|f|^{2k}),\]
for every $k\in \NN.$ Fix $k\in \NN,$ then
\[|\tr\otimes \Tr(|\sigma_{i}(f)|^{2k}-|A_{i}|^{2k})|\leq m\||\sigma_{i}(f)|^{2k}-|A_{i}|^{2k}\|_{2}.\]
We claim that
\[\||\sigma_{i}(f)|^{2k}-|A_{i}|^{2k}\|_{2}\to 0,\]
this will complete the proof, since we know the Lemma holds when $A_{i}=\sigma_{i}(f).$

	For this, note that
\[\||\sigma_{i}(f)|^{2k}-|A_{i}|^{2k}\|_{2}\leq \|\sigma_{i}(f)\|_{\infty}^{2}\||\sigma_{i}(f)|^{2(k-1)}-|A_{i}|^{2(k-1)}\|_{2}+\||A_{i}|^{2(k-1)}\|_{\infty}\||\sigma_{i}(f)|^{2}-|A_{i}|^{2}\|_{2}\]
and so it suffices to handle the case $k=1.$ We use that if $T\in M_{n,m}(M_{d_{i}}(\CC)),S\in M_{m,n}(M_{d_{i}}(\CC)),$ then
\[\|ST\|_{2}\leq \|S\|_{\infty}\|T\|_{2},\]
\[\|ST\|_{2}\leq \sqrt{\frac{n}{m}}  \|T\|_{\infty}\|S\|_{2},\]
\[\|T\|_{2}=\sqrt{\frac{n}{m}}\|T^{*}\|_{2}.\]
Hence,
\[\||\sigma_{i}(f)|^{2}-|A_{i}|^{2}|_{2}\leq \|\sigma_{i}(f)\|_{\infty}\|\sigma_{i}(f)-A_{i}\|_{2}+\|\sigma_{i}(f)-A_{i}\|_{2}\|A_{i}\|_{\infty}\to 0.\]

\end{proof}

	The proof of the next Lemma contains the essential ideas behind our use of sofic approximations to compute von Neumann dimension.

\begin{lemma}[The von Neumann Dimension Lemma]\label{L:vnDLemma} Let $\Gamma$ be a countable discrete sofic group and $\sigma_{i}\colon \Gamma\to S_{d_{i}}$ a sofic approximation.  Extend $\sigma_{i}$ to $\RR(\Gamma)\to M_{d_{i}}(\RR)$ by $\sigma_{i}(f)=\sum_{g\in\Gamma}\widehat{f}(g)\sigma_{i}(g).$ Further extend  $\sigma_{i}$ to $M_{n,m}(\RR(\Gamma))\to M_{n,m}(M_{d_{i}}(\RR))$ by $\sigma_{i}(f)_{jk}=\sigma_{i}(f_{jk}).$  Fix $f\in M_{n,m}(\RR(\Gamma)).$

	The following quantities are equal:

\[(i) \dim_{R(\Gamma)}(\ker \rho(f)),\]

\[(ii) \liminf_{\varepsilon\to 0}\inf_{\eta}\liminf_{i\to \infty}\frac{\log S_{\varepsilon}(\chi_{[0,\eta]}(|\sigma_{i}(f)|)\Ball(\ell^{2}_{\RR}(d_{i},\mu_{d_{i}})^{\oplus m}))}{d_{i}\log(1/\varepsilon)},\]

\[(iii)\limsup_{\varepsilon\to 0}\inf_{\eta}\limsup_{i\to \infty}\frac{\log S_{\varepsilon}(\chi_{[0,\eta]}(|\sigma_{i}(f)|)\Ball(\ell^{2}_{\RR}(d_{i},\mu_{d_{i}})^{\oplus m}))}{d_{i}\log(1/\varespilon)},\]

\[(iv)\inf_{\eta}\liminf_{\varepsilon\to 0}\liminf_{i\to \infty}\frac{\log S_{\varepsilon}(\chi_{[0,\eta]}(|\sigma_{i}(f)|)\Ball(\ell^{2}_{\RR}(d_{i},\mu_{d_{i}})^{\oplus m}))}{d_{i}\log(1/\varespilon)},\]
\[(v)\inf_{\eta}\limsup_{\varepsilon\to 0}\limsup_{i\to \infty}\frac{\log S_{\varepsilon}(\chi_{[0,\eta]}(|\sigma_{i}(f)|)\Ball(\ell^{2}_{\RR}(d_{i},\mu_{d_{i}})^{\oplus m}))}{d_{i}\log(1/\varespilon)}.\]

\end{lemma}

\begin{proof} Let $d_{\eta,i}=\tr\otimes \Tr(\chi_{[0,\eta]}(|\sigma_{i}(f)|)),$ by a standard volume packing argument
\[\varepsilon^{-d_{i}d_{\eta,i}}\leq S_{\varepsilon}(\chi_{[0,\eta]}(|\sigma_{i}(f)|)\Ball(\ell^{2}_{\RR}(d_{i},\mu_{d_{i}})^{\oplus m})\leq \left(\frac{3+3\varepsilon}{\varepsilon}\right)^{d_{i}d_{\eta,i}}.\]
Since $\mu_{|\sigma_{i}(f)|}\to \mu_{|f|}$ weak$^{*}$
\[\inf_{\eta}\limsup_{i\to \infty}d_{\eta,i}=\mu_{|f|}(\{0\})\]
and
\[\inf_{\eta}\liminf_{i\to \infty}d_{\eta,i}=\mu_{|f|}(\{0\}).\]

	By direct computation,  for any $\alpha\in M_{n,m}(\CC(\Gamma)),$ we have
	\[\Tr\otimes \tau_{R(\Gamma)}((\rho(\alpha)^{*}\rho(\alpha))^{k})=\Tr\otimes\tau_{L(\Gamma)}((\lambda(\alpha)^{*}\lambda(\alpha))^{k})\]
for any $k\in \NN.$ That is,
\[\int_{[0,\infty)} t^{2k}\,d\mu_{|\rho(\alpha)|}(t)=\int_{[0,\infty)}t^{2k}\,d\mu_{|\alpha|}(t)\]
for all $\alpha\in M_{n,m}(\CC(\Gamma)).$ Since $\mu_{|\rho(\alpha)|},\mu_{|\alpha|}$ are both compactly supported, this implies that
\[\mu_{|\rho(\alpha)|}=\mu_{|\alpha|}\]
for all $\alpha\in M_{n,m}(\CC(\Gamma)).$ As the orthogonal projection onto $\ker(\rho(f))$ is $\chi_{\{0\}}(|\rho(f)|),$ we have that
\[\mu_{|f|}(\{0\})=\mu_{|\rho(f)|}(\{0\})=\Tr\otimes \tau_{R(\Gamma)}(\chi_{\{0\}}(|\rho(f)|))=\dim_{R(\Gamma)}(\ker \rho(f)).\]
The first equality following from what we just observed, the second by definition, and the last by the original definition of von Neumann dimension due to Murray and von Neumann as explained in the remarks before Lemma \ref{L:hilred}.

\end{proof}

\subsection{Preliminaries on Mean Dimension}\label{SS:meandimension}

	We recall the definitions from \cite{Li}. For notational purposes, if $\rho$ is a pseudometric on a space $X,$ and $p\in [1,\infty]$ we denote $\rho_{p}$ the pseudometric on $X^{k}$ defined by
\[\rho_{p}(x,y)=\left(\frac{1}{k}\sum_{j=1}^{k}\rho(x_{j},y_{j})^{p}\right)^{1/p},\mbox{ if $p<\infty$}\]
\[\rho_{\infty}(x,y)=\max_{1\leq j\leq k}\rho(x(j),y(j)).\]
We will typically not need to indicate the dependence on $k$ in the notation, as it will be implicit.

\begin{definition} \emph{Let $\Gamma$ be a countable discrete sofic group and $\Sigma$ a sofic approximation of $\Gamma.$ Let $X$ be a compact metrizable space and $\Gamma\actson X.$ Fix a pseudometric $\rho$ on $X.$ For $F\subseteq \Gamma$ finite, $i\in \NN,\delta>0,$ we define $\Map(\rho,F,\delta,\sigma_{i})$ to be the set of all maps $\phi\colon \{1,\cdots,d_{i}\}\to X,$ such that $\rho_{2}(\phi\circ \sigma_{i}(g),g\phi)<\delta,$ for all $g\in F.$}\end{definition}

	Though it will not be used for most of the paper, we recall the notion of mean dimension. This was originally defined by Lindenstrauss and Weiss in \cite{LindWeiss} for amenable groups, and extended to sofic groups by Hanfeng Li in \cite{Li}. We recall the definition for sofic groups.

\begin{definition} \emph{Let $(X,d)$ be a compact metric space. For an open cover $\mathcal{U}$ of $X,$ set}
\[\ord(\mathcal{U})=\sup_{x\in X}-1+\sum_{U\in \mathcal{U}}\chi_{U}(x).\]
\emph{Let $\mathcal{D}(\mathcal{U})$ be $\inf_{\mathcal{V}\succeq \mathcal{U}}\ord(\mathcal{V})$ where $\mathcal{V}\succeq \mathcal{U}$ means that $\mathcal{V}$ refines $\mathcal{U}.$ Now finally, set $\dim(X)=\sup_{\mathcal{U}}\mathcal{D}(\mathcal{U}),$ and call this the }covering dimension\emph{ of $X.$}
\end{definition}

\begin{definition} \emph{Let $X$ be a compact metrizable space with a continuous pseudometric $d.$ We set $\wdim_{\varepsilon}(X,d)$ to be the smallest $k$ such that there is a $k$-dimensional compact topological space $K$ and a continuous map $f\colon X\to K$ such that for all $x,y\in X$ with $f(x)=f(y),$ we have $d(f(x),f(y))<\varepsilon.$}\end{definition}

\begin{definition}\emph{ Let $\Gamma$ be a countable discrete group acting by homeomorphisms on a compact metrizable space $X.$ A continuous pseudometric $\rho$ on $X$ is said to be} dynamically generating \emph{ if for all $x\ne y,$ there is a $g\in \Gamma$ so that $\rho(gx,gy)>0.$}\end{definition}

\begin{definition}\emph{ Let $\Gamma$ be a countable discrete sofic group acting by homeomorphisms on a compact metrizable space $X,$ let $\rho$ be a dynamically generating continuous pseudometric on $X,$ and $\rho'$ a compatible metric on $X.$ Let $\Sigma$ be a sofic approximation of $\Gamma.$  Let $\varepsilon>0.$  For $F\subseteq \Gamma$ finite and,$\delta>0$ set}
\[\mdim_{\Sigma}(\varepsilon,F,\delta,\Gamma)=\limsup_{i\to \infty}	\frac{\wdim_{\varepsilon}(\Map(\rho,F,\delta,\sigma_{i}),\rho'_{\infty})}{d_{i}},\]
\[\mdim_{\Sigma}(\varepsilon)=\inf_{\substack{F\subseteq \Gamma\mbox{\emph{ finite}}\\\delta>0}}\mdim_{\Sigma}(\varepsilon,F,\delta,\Gamma),\]
\[\mdim_{\Sigma}(X,\Gamma)=\sup_{\varespilon>0}\mdim_{\Sigma}(\varepsilon).\]
\end{definition}

	We call $\mdim_{\Sigma}(X,\Gamma)$ the mean dimension of $\Gamma \actson X.$
It is proved in \cite{Li} that this definition is independent of the choice of pseudometric $\rho$ and metric $\rho'.$

	The above definition is different than the definition in \cite{Li}, but the two version are equivalent. For one direction, one considers the Lebesgue number with respect to $\rho'$ of an open cover and applies Proposition 2.4 in \cite{LindWeiss}. For the other direction, one uses an open cover by balls of radius $\varepsilon$ in the $\rho'$ pseudometric and applies Proposition 2.4 in \cite{LindWeiss}. It will be easier for us to work with $\varepsilon$-widths, so that is what we will do throughout the paper.The closely related notion of \emph{metric} mean dimension is the one of main importance for our results. We consider an extended notion which will be relevant for our purposes.

\begin{definition}\emph{ Let $\Gamma$ be a countable discrete sofic group acting by homeomorphisms on a compact metrizable space $X,$ let $\rho$ be a dynamically generating continuous pseudometric on $X,$ and $\Sigma$ a sofic approximation of $\Gamma.$ For $F\subseteq \Gamma$ finite, $\varepsilon>0$ and $\delta>0,$ set}\end{definition}
\[\mdim_{\Sigma,p}(F,\delta,\varepsilon,\rho)=\limsup_{i\to \infty}\frac{\log S_{\varepsilon}(\Map(\rho,F,\delta,\sigma_{i}),\rho_{p})}{d_{i}\log(1/\varepsilon)},\]
\[\mdim_{\Sigma,p}(\varespilon,\rho)=\inf_{\substack{F\subseteq \Gamma \mbox{\mbox{finite}}\\ \delta>0}}\mdim_{\Sigma,p}(F,\delta,\varepsilon,\rho),\]
\[\mdim_{\Sigma,p}(X,\rho)=\liminf_{\varepsilon\to 0}\mdim_{\Sigma,p}(\varespilon,\rho).\]
Finally, we define the sofic $p$-metric mean dimension by
\[\mdim_{\Sigma,M,p}(X,\Gamma)=\inf_{\rho}\mdim_{\Sigma,p}(X,\rho),\]
with the infimum being over all dynamically generating continuous pseudometrics $\rho$ on $X.$ For $p=\infty,$ this is metric mean dimension as defined in \cite{Li}, and this will sometimes be denoted $\mdim_{\Sigma,M}(X,\Gamma)$ instead of $\mdim_{\Sigma,M,p}(X,\Gamma).$  Li actually works with $N_{\varepsilon}(A,d)$ which is the largest cardinality of an $\varepsilon$-separated subset of $A$ (we say $B\subseteq A$ is $\varepsilon$-separated if $d(x,y)\geq \varespilon$ if $x,y\in B$ and $x \ne y$) but it is easy to show that
\[N_{2\varepsilon}(A,d)\leq S_{\varespilon}(A,d)\leq N_{\varepsilon}(A,d),\]
 so working with $\varepsilon$-dense subsets instead of $\varepsilon$-separated subsets is simply a matter of taste. Our arguments will work better with $\varepsilon$-dense subsets, so we will use that throughout the paper.

 In \cite{Li}, it is shown that metric mean dimension  equals the infimum as above over all metrics $\rho$ compatible with the topology on $X.$ We can show the same for $p$-metric mean dimension. Before doing this, however we would like to isolate a key Lemma. This Lemma appears as Lemma 2.3 in \cite{KerrLi2}. Using this Lemma, one can also rephrase the proofs of Theorem 2.6 and Theorem 4.5 in \cite{KLi} in a manner entirely similar to the proof of invariance in section 2 of \cite{Me}, and in fact this was the inspiration for the proof in section 2 of \cite{Me}.

\begin{lemma}\label{L:heart} Let $\Gamma$ be a countable discrete sofic group acting by homeomorphisms on a compact metrizable space $X,$ and let $\Sigma$ be a  sofic approximation of $\Gamma.$ Let $\rho,\rho'$ be two dynamically generating continuous pseudometrics on $X.$ Then for all $F\subseteq \Gamma$ finite, $\delta>0,$ there is an $F'\subseteq \Gamma$ finite, a $\delta'>0$ so that for all large $i,$
\[\Map(\rho',F',\delta',\sigma_{i})\subseteq \Map(\rho,F,\delta,\sigma_{i}).\]
\end{lemma}

	The case $q=\infty$ of the next Lemma is more difficult, and is Lemma 4.4 in \cite{Li}. We will only need the case $1\leq q<\infty.$

\begin{lemma} Let $\Gamma$ be a countable discrete sofic group acting by homeomorphisms on a compact metrizable space $X,$ and let $\Sigma$ be a sofic approximation of $\Gamma,$ then for all $1\leq q<\infty,$
\[\mdim_{\Sigma,M,q}(X,\Gamma)=\inf_{\rho}\mdim_{\Sigma,q}(X,\rho),\]
where the infimum is over all compatible metrics $\rho$ on $X.$
\end{lemma}

\begin{proof} This is entirely similar to the proof of Lemma 4.4 in \cite{Li}, again we repeat the proof for the reader's convenience. Fix a dynamically generating continuous pseudometric $\rho$ on $X.$ Enumerate the elements of $\Gamma$ as $s_{1},s_{2},\cdots$  with $s_{1}=e,$ and let $\rho'$ be the metric on $X$ defined by
\[\rho'(x,y)=\sum_{n=1}^{\infty}\frac{1}{2^{n}}\rho(s_{n}x,s_{n}y),\]
then $\rho'$ is a compatible metric on $X.$ To prove the Lemma, it is enough to show that
\[\mdim_{\Sigma,q}(X,\rho)=\mdim_{\Sigma,q}(X,\rho').\]

	By the preceding Lemma, for every $F'\subseteq \Gamma$ finite, $\delta'>0$ there is an $F\subseteq \Gamma$ finite, $\delta>0$ so that
\[\Map(\rho,F,\delta,\sigma_{i})\subseteq \Map(\rho',F',\delta',\sigma_{i})\]
for all large $i.$ As $\rho\leq \rho',$ we have that
\[S_{2\varepsilon}(\Map(\rho,F,\delta,\sigma_{i}),\rho_{q})\leq S_{2\varepsilon}(\Map(\rho,F,\delta,\sigma_{i}),\rho'_{q})\leq S_{\varepsilon}(\Map(\rho,F',\delta',\sigma_{i}),\rho'_{q}).\]
Thus
\[\mdim_{\Sigma,q}(X,\rho)\leq \mdim_{\Sigma,q}(X,\rho').\]

	For the reverse inequality, set $M=\sup_{x,y\in X}\rho(x,y).$ Fix $\varepsilon>0,$ and take $k\in \NN$ so that $2^{-k}M<\varepsilon.$ Let $F\supseteq \{s_{1},\cdots,s_{k}\}$ be a finite subset of $\Gamma,$ and $\delta>0.$ We may assume that $\delta<1.$ Then by the preceding Lemma, there is a finite set $F'\subseteq \Gamma,$ and $\delta'>0$ so that for all large $i$ we have $\Map(\rho,F,\delta,\sigma_{i})\supseteq \Map(\rho',F',\delta',\sigma_{i}).$
For $x,y\in X,$
\[\rho'(x,y)\leq \varepsilon+\sum_{n=1}^{k}\frac{1}{2^{n}}\rho(s_{n}x,s_{n}y).\]
Now fix $\phi,\psi\in \Map(\rho,F,\delta,\sigma_{i}).$ By Minkowski's inequality,
\[\rho'_{q}(\phi,\psi)\leq \varespilon+ \sum_{n=1}^{k}\frac{1}{2^{n}}\rho_{q}(s_{n}\phi,s_{n}\psi).\]

	Since $\phi\in \Map(\rho,F,\delta,\sigma_{i})$ for $1\leq n\leq k$ we have that $\rho(s_{n}\phi(j),\phi(\sigma_{i}(s_{n})j))<\sqrt{\delta}$ for at least $(1-\delta)d_{i}$ of the $j.$  Thus
\[\rho_{q}^{q}(s_{n}\phi,\phi\circ \sigma_{i}(s_{n}))\leq \delta^{q/2}M(1-\delta)+\delta M^{q},\]
and this can be made less that $\varespilon^{q}$ if $\delta$ is small enough. Thus
\[\rho_{q}(s_{n}\phi,s_{n}\psi)\leq 2\varepsilon+\rho_{q}(\phi\circ \sigma_{i}(s_{n}),\psi\circ \sigma_{i}(s_{n}))=2\varepsilon+\rho_{q}(\phi,\psi).\]
Hence,
\[\rho'_{q}(\phi,\psi)\leq 3\varepsilon+\rho_{q}(\phi,\psi).\]
Therefore,
\[S_{8\varepsilon}(\Map(\rho',F',\delta',\sigma_{i}),\rho'_{p})\leq S_{\varespilon}(\Map(\rho,F,\delta,\sigma_{i}),\rho_{q}),\]
and this is enough to complete the proof.

\end{proof}

\begin{lemma} Let $\mu_{n}$ be the uniform probability measure on $\{1,\cdots,n\}.$ There is a $C>0,$ so that if $m,n\in \NN,$ and $S\subseteq \{1,\cdots,m\}\times \{1,\cdots,n\},$ then
\[\vol_{\RR^{S}}(\Ball(l^{\infty}_{\RR}(m,l^{1}_{\RR}(n,\mu_{n})))\cap \RR^{S})^{1/|S|}\leq C\frac{nm}{|S|}.\]

\end{lemma}

\begin{proof} It is known that there is a  $C>0$  so that
\[\vol(\Ball(l_{\RR}^{1}(k,\mu_{k})))\leq C^{k}\]
for all $k\in \NN$ (this is implied by the calculation on page 11 of \cite{Pis}).
	
	Write $S$ as
\[S=\bigcup_{j}\{j\}\times B_{j},\]
with $B_{j}\subseteq \{1,\cdots,n\}.$ Then
\[\Ball(l^{\infty}_{\RR}(m,l^{1}_{\RR}(n,\mu_{n})))\cap \RR^{S}=\prod_{j=1}^{m}\Ball(l^{1}_{\RR}(B_{j},\mu_{n}))=\prod_{j=1}^{m}\frac{n}{|B_{j}|}\Ball(l^{1}_{\RR}(B_{j},\mu_{B_{j}})).\]
Hence
\begin{align*}
\vol_{\RR^{S}}(\Ball(l^{\infty}_{\RR}(m,l^{1}_{\RR}(n,\mu_{n}))\cap \RR^{S}))^{1/|S|}&\leq C\prod_{j=1}^{m}\left(\frac{n}{|B_{j}|}\right)^{\frac{|B_{j}|}{|S|}}\\
&\leq C\sum_{j=1}^{m}\frac{n}{|S|}\\
&=C\frac{nm}{|S|},
\end{align*}
by the Arithmetic-Geometric Mean inequality.

\end{proof}

\begin{lemma} Let $X$ be a compact metrizable space with compatible metric $\rho,$ and  let $\mathcal{U}$ be a finite open cover of $X.$  Let $f_{U}\colon X\to [0,1],$ be a family of Lipschitz functions indexed by elements of $\mathcal{U}$ such that
\[\max_{U\in \mathcal{U}}f_{U}(x)=1,\]
for all $x\in X$ and
\[f_{U}(x)=0\mbox{ for $x\in X\setminus U$}\]
for all $U\in\mathcal{U}.$  (See \cite{Li} Lemma 6.2 for the existence of such a family). Define
\[\Psi_{i}\colon X^{d_{i}}\to [0,1]^{\{1,\dots,d_{i}\}\times \mathcal{U}},\]
by
\[\Psi_{i}(x)(j,U)=f_{U}(x(j)).\]
For every $\alpha,\eta>0,$ there is a $0<\varepsilon_{0}$ satisfying the following. Suppose $0<\varepsilon<\varepsilon_{0},$ and $\Xi_{i}\subseteq X^{d_{i}}$ satisfies
\[\alpha\geq\limsup_{i\to \infty}\frac{\log S_{\varepsilon}(\Xi_{i},\rho_{1})}{d_{i}\log (1/\varespilon)}.\]
	Then for all large $i,$ we can find a $\xi\in (0,1)^{[d_{i}]\times \mathcal{U}},$ so that for all  $S\subseteq \{1,\cdots,d_{i}\}\times \mathcal{U},$ with $|S|\geq (\alpha+\eta)d_{i},$
\[\xi\big|_{S}\notin \Psi_{i}(\Xi_{i})\big|_{S}.\]
\end{lemma}

\begin{proof} This proof is essentially the same as Lemma 6.3 in \cite{Li}, but we shall repeat the proof for the reader's convenience. The main tool to adapt the proof is the preceding Lemma. Concretely, let $L>1$ be a bound for the Lipschitz seminorms of $f_{U}$ for $U\in \mathcal{U}.$ Let $\varepsilon_{0}$  be sufficiently small in a manner to be determined later. Let $0<\varepsilon<\varepsilon_{0},$ and $\Xi_{i}\subseteq X^{d_{i}}$ with
\[\alpha\geq\limsup_{i\to \infty}\frac{\log S_{\varepsilon}(\Xi_{i},\rho_{1})}{d_{i}\log (1/\varespilon)}.\]
Then for all large $i,$
\[S_{\varepsilon}(\Xi_{i},\rho_{1})\leq \varepsilon^{-d_{i}(\alpha+\frac{\eta}{(2015)!})}.\]

	Thus $\Psi_{i}$ can be covered by at most $\varepsilon^{-d_{i}(\alpha+\frac{\eta}{(2015)!})}$ balls of radius $L\varepsilon$ in the $\|\cdot\|_{l^{\infty}(\mathcal{U},l^{1}(d_{i},\mu_{d_{i}}))}$-norm. Let $C$ be as in the above Lemma, (note $C>1$), and fix $S\subseteq \mathcal{U}\times \{1,\cdots,d_{i}\}$ with $|S|\geq (\alpha+\eta)d_{i}.$ Then
\[\Psi_{i}(\Xi_{i})\big|_{S}\subseteq [0,1]^{S},\]
can be covered by at most  $\varepsilon^{-d_{i}(\alpha+\frac{\eta}{(2015)!})}$ balls of radius $L\varepsilon$ in the $\|\cdot\|_{l^{\infty}(\mathcal{U},l^{1}(d_{i},\mu_{d_{i}}))}$-norm. By the preceding Lemma,
\begin{align*}
\vol_{\RR^{S}}(\Psi_{i}(\Xi_{i})\big|_{S})\leq &\left(\frac{|\mathcal{U}|CL\varepsilon}{\alpha+\eta}\right)^{|S|}\varepsilon^{-d_{i}(\alpha+\frac{\eta}{(2015)!})}\\
&\leq \left(\frac{|\mathcal{U}|CL}{\alpha+\eta}\right)^{|S|}\varepsilon^{d_{i}(\eta-\frac{\eta}{(2015)!})}.
\end{align*}
Set
\[A=\max\left(\frac{|\mathcal{U}|CL}{\alpha+\eta},1\right).\]
Let $Q$ be the set of all $\xi\in (0,1)^{\{1,\dots,d_{i}\}\times \mathcal{U}}$ so that $\xi\big|_{S}\in \Psi_{i}(\Xi_{1})\big|_{S}$ for  some $S$ with $|S|\geq (\alpha+\eta)d_{i}.$ Then
\[\vol(Q)^{1/d_{i}}\leq A^{|\mathcal{U}|}\varespilon^{\eta-\frac{\eta}{(2015)!}}2^{|\mathcal{U}|}.\]
If $\varepsilon_{0}$ is sufficiently small (depending only on $\mathcal{U},\alpha,\eta,L$) then $\vol(Q)<1$ for all large $i.$

\end{proof}

\begin{lemma}\label{L:subsetsofproducts} Let $X$ be a compact metrizable space with a compatible metric $\rho,$ and let $\rho'$ be a continuous pseudometric on $X.$ Let $\delta>0.$ Let $d_{i}$ be a sequence of integers going to infinity, let $J$ be a  set and  $\Xi_{i}(\beta)\subseteq X^{d_{i}},$ for $\beta\in J.$ Then
\[\inf_{\beta}\limsup_{i\to \infty}\frac{\wdim_{\delta}(\Xi_{i}(\beta),\rho'_{\infty})}{d_{i}}\leq \liminf_{\varepsilon\to 0}\inf_{\beta} \limsup_{i\to \infty}\frac{\log S_{\varepsilon}(\Xi_{i}(\beta),\rho_{1})}{d_{i}\log (1/\varepsilon)}.\]
\end{lemma}

\begin{proof}  Set
\[D= \liminf_{\varepsilon\to 0}\inf_{\beta}\limsup_{i\to \infty}\frac{\log S_{\varepsilon}(\Xi_{i}(\beta),\rho_{1})}{d_{i}\log (1/\varepsilon)}.\]
Let $\mathcal{U}$ be an open cover of $X$ consisting of balls of radius $\delta/2$ in the $\rho'$ pseudometric. For $k\in \NN,$ we define
\[\mathcal{U}^{k}=\{U_{1}\times U_{2}\times U_{3}\times \cdots U_{k}:U_{j}\in \mathcal{U},1\leq j\leq k\}.\]
 Let $\eta>0,$ and let $\varepsilon_{0}$ be as in the above Lemma for $\alpha=D+2\eta.$ Choose $\varespilon_{0}>\varepsilon>0$ so that
\[\inf_{\beta} \limsup_{i\to \infty}\frac{\log S_{\varepsilon}(\Xi_{i}(\beta),\rho_{1})}{d_{i}\log (1/\varepsilon)}\leq D+\eta.\]
Thus we can choose $\beta_{0}$ so that
\[ \limsup_{i\to \infty}\frac{\log S_{\varepsilon}(\Xi_{i}(\beta_{0}),\rho_{1})}{d_{i}\log (1/\varepsilon)}\leq D+2\eta.\]

	By our choice of $\varepsilon_{0},$ we may apply the preceding Lemma, and follow the proof in \cite{LindWeiss} of Theorem 4.2 and Lemma 4.4 ( see also Lemmas 6.4 and 6.5 in \cite{Li}),  to find a continuous map (defined for all large $i$)
\[f\colon \Xi_{i}(\beta_{0})\to \bigcup_{\substack{|S|\leq (D+3\eta)d_{i},\\ S\subseteq \mathcal{U}\times \{1,\cdots,d_{i}\}}}[0,1]^{S}\times \{0,1\}^{(\mathcal{U}\times \{1,\cdots,d_{i}\})\setminus S}\subseteq [0,1]^{\mathcal{U}\times \{1,\cdots,d_{i}\}},\]
such that for every $z\in [0,1]^{\mathcal{U}\times \{1,\dots,d_{i}\}}$ there is some $U\in\mathcal{U}^{d_{i}}$ with $f^{-1}(\{z\})\subseteq U.$  Since
\[\bigcup_{\substack{|S|\leq (D+3\eta)d_{i},\\ S\subseteq \mathcal{U}\times \{1,\cdots,d_{i}\}}}[0,1]^{S}\times \{0,1\}^{(\mathcal{U}\times \{1,\cdots,d_{i}\})\setminus S}\subseteq [0,1]^{\mathcal{U}\times \{1,\cdots,d_{i}\}},\]
has dimension at most $(D+3\eta)d_{i},$ we find that
\[\limsup_{i\to \infty}\frac{1}{d_{i}}\wdim_{\delta}(\Xi_{i}(\beta_{0}),\rho'_{\infty})\leq D+3\eta.\]
Thus
\[\inf_{\beta_{0}}\limsup_{i\to \infty}\frac{1}{d_{i}}\wdim_{\delta}(\Xi_{i}(\beta_{0}),\rho'_{\infty})\leq D+3\eta,\]
and as $\eta>0$ was arbitrary we are done.

\end{proof}

Proceeding as in \cite{Li}, we have the following corollary.
\begin{cor}\label{C:mdimcomparison} Let $\Gamma$ be a countable discrete sofic group acting by homeomorphisms on a compact metrizable space $X.$ For any sofic approximation $\Sigma$ and $1\leq p\leq q\leq \infty$ we have
\[\mdim_{\Sigma}(X,\Gamma)\leq \mdim_{\Sigma,M,p}(X,\Gamma)\leq \mdim_{\Sigma,M,q}(X,\Gamma).\]

\end{cor}

\begin{proof} The right-hand inequalities are trivial. For the left most inequality, it is enough to handle the case $p=1.$ Fix a compatible metric $\rho$ on $X.$  Taking $J$ to be the set of tuples $(F,\delta),$ with $F\subseteq \Gamma$ finite,  and $\delta>0,$ the preceding Lemma proves the left-most inequality.

\end{proof}

While in general it is unclear if all these versions of metric mean dimension agree, there is one simple case where they do agree. Recall that if $X$ is a compact topological space with a continuous pseudometric $\rho,$ then the  packing dimension of $X$ is defined as
\[\limsup_{\varepsilon\to 0}\frac{\log S_{\varespilon}(X,\rho)}{\log(1/\varespilon)}.\]

\begin{proposition}\label{P:finitepack} Let $\Gamma$ be a countable discrete sofic group with sofic approximation $\Sigma$ acting by homeomorphisms on a compact metrizable space $X.$ If $\rho$ is a dynamically generating continuous pseudometric on $X$ so that $(X,\rho)$ has finite packing dimension, then
\[\mdim_{\Sigma,1}(X,\rho)=\mdim_{\Sigma,\infty}(X,\rho).\]
\end{proposition}

\begin{proof}

It is easy to see that
\[\mdim_{\Sigma,\infty}(X,\rho)\geq \liminf_{\varepsilon\to 0}\inf_{F,\delta}\limsup_{i\to \infty}\frac{\log S_{\varepsilon}(\Map(\rho,F,\delta,\sigma_{i}),\rho_{1})}{d_{i}\log (1/\varepsilon)}.\]

	For the converse, fix $0<\alpha<1.$ If $\phi,\psi$ are maps from $\{1,\cdots,d_{i}\}$ to $ X$ and $\rho_{1}(\phi,\psi)<\varepsilon,$ then $\rho(\phi(j),\psi(j))<\varepsilon^{\alpha}$ on a set of cardinality at least $(1-\varepsilon^{1-\alpha})d_{i}.$ The set of all subsets of $\{1,\cdots,d_{i}\}$ of cardinality at most $\varepsilon^{1-\alpha}d_{i}$ is at most
\[\varepsilon^{1-\alpha}d_{i}\binom{d_{i}}{\lfloor{\varepsilon^{1-\alpha}d_{i}\rfloor}}\leq C\varepsilon^{1-\alpha}d_{i}\frac{1}{(\varepsilon^{1-\alpha})^{d_{i}\varepsilon^{1-\alpha}}((1-\varepsilon^{1-\alpha}))^{d_{i}(1-\varepsilon^{1-\alpha})}}\]
for some $C>0,$ by Stirling's Formula. Thus
\begin{align*}
\frac{\alpha\log(1/\varepsilon)-\log(2)}{\log(1/\varepsilon)}\mdim_{\Sigma,\infty}({2\varepsilon^{\alpha}},\rho)&\leq \mdim_{\Sigma,1}(\varepsilon,\rho)+\varespilon^{1-\alpha}\frac{\log S_{\varepsilon^{\alpha}}(X,\rho)}{\log(1/\varepsilon)}\\
&+(1-\alpha)\varepsilon^{1-\alpha}-(1-\varepsilon^{1-\alpha})\frac{\log(1-\varepsilon^{1-\alpha})}{\log(1/\varepsilon)}.
\end{align*}
Since $(X,\rho)$ has finite packing dimension, we find that
\[\alpha\mdim_{\Sigma,\infty}(X,\rho)\leq \mdim_{\Sigma,1}(X,\rho),\]
and the proof is completed by letting $\alpha\to 1.$

\end{proof}

A dynamically generating continuous pseudometric with finite packing dimension always exists if $X\subseteq Y^{\Gamma}$ with $Y$ a compact manifold. For algebraic actions corresponding to finitely generated modules, one can take $Y=\TT^{n}$ for some $n,$ and this is how we will apply the above Proposition.

\subsection{p-Metric Mean Dimension For Amenable Groups}

	In this subsection, we give an alternate formulation for $p$-metric mean  dimension for actions of amenable groups. We shall show that for actions of amenable groups, $p$-metric mean dimension is actually given by an $\ell^{p}$-analogue of metric mean dimension as defined by Lindenstrauss and Weiss. In particular, we will show that $p$-metric mean dimension for amenable groups is independent of the choice of sofic approximation. For this, we use the following Theorem of Gabor Elek and Endre Szabo, which gives a uniqueness result for sofic approximations of an amenable group.

\begin{theorem}[Theorem 2 in \cite{ESZ2}]\label{T:therecanbeonlyone!} Let $\Gamma$ be a countable discrete amenable group, and let $d_{i}$ be a sequence of positive integers tending to $\infty.$ Suppose that $\Sigma=(\sigma_{i}\colon \Gamma\to S_{d_{i}}),\Sigma'=(\sigma'_{i}\colon \Gamma\to S_{d_{i}})$ are two sofic approximations. Then there are $\tau_{i}\in S_{d_{i}}$ so that
\[\frac{|\{j:\tau_{i}\sigma_{i}(g)\tau_{i}^{-1}(j)=\sigma_{i}'(g)(j)\}|}{d_{i}}\to 1\]
for all $g\in \Gamma.$\end{theorem}

\begin{proposition}\label{P:independence} Let $\Gamma$ be a countable discrete amenable group. Let $X$ be a compact metrizable space and $\Gamma\actson X$ by homeomorphisms. Let $\rho$ be a dynamically generating continuous pseudometric on $X.$ Let $1\leq p<\infty,$ and $\Sigma,\Sigma'$ be two sofic approximations of $\Gamma.$ Then
\[\mdim_{\Sigma,p}(X,\rho)=\mdim_{\Sigma',p}(X,\rho).\]
\end{proposition}
\begin{proof} We set up some notation. For two functions (again not necessarily homomorphisms) $\phi\colon \Gamma\to S_{n},\psi\colon \Gamma\to S_{m}$ we let $\phi\oplus \psi\colon \Gamma\to S_{n+m}$ be defined by
\[\phi\oplus \psi(j)=\begin{cases}
\phi(j),\textnormal{ if $1\leq j\leq n$}\\
n+\psi(j-n)\textnormal{ if $n+1\leq j\leq n+m$}.
\end{cases}\]

	We use $\phi^{\oplus k}$ for the $k$-fold direct sum of $\phi.$ Let $t_{l}\colon \Gamma\to S_{l}$ be the trivial homomorphism. We leave it as an exercise to show that we may find an  increasing function
\[l\colon \NN\to \NN,\]
with
\[\lim_{n\to \infty}l(n)=\infty,\]
 and  functions
\[r\colon \NN\to \NN\]
\[k\colon \NN\to \NN\]
so that
\[\frac{r(i)}{k(i)d_{l(i)}}\to_{i\to\infty} 0,\]
and
\[d_{i}'=d_{l(i)}k(i)+r(i).\]
Set
\[\widetilde{\Sigma}=(\sigma_{l(i)}^{\oplus k(i)}\oplus t_{r(i)})_{i\geq 1}.\]
By the Theorem of Elek-Szabo, there are $\tau_{i}\in S_{d_{i}'}$ so that
\[\frac{|\{j:\tau_{i}\sigma_{i}'(g)\tau_{i}^{-1}(j)=(\sigma_{l(i)}^{\oplus k(i)}\oplus t_{r(i)})(g)(j)\}|}{d_{i}'}\to 1.\]
From this it is easy to see that for any finite $F\subseteq\Gamma$ and $\delta>0,$ we have for any $0<\delta'<\delta$ that
\[\{\phi\circ \tau_{i}^{-1}:\phi\in \Map(\rho,F,\delta',\sigma_{i}')\}\subseteq \Map(\rho,F,\delta,\sigma_{l(i)}^{\oplus k(i)}\oplus t_{r(i)})\]
for all large $i.$ Similarly, for every $F\subseteq\Gamma$ finite, and $\delta'>0$ we have for any $0<\delta<\delta'$ that
\[\{\phi\circ \tau_{i}:\phi\in \Map(\rho,F,\delta,\sigma_{l(i)}^{\oplus k(i)}\oplus t_{r(i)})\}\subseteq \Map(\rho,F,\delta',\sigma_{i}')\]
for all large $i.$ Hence
\[\mdim_{\Sigma',p}(X,\rho)=\mdim_{\widetilde{\Sigma},p}(X,\rho).\]
Thus we may assume that
\[\sigma_{i}'=\sigma_{l(i)}^{\oplus k(i)}\oplus t_{r(i)}.\]

	Set
\[\sigma_{i}''=\sigma_{l(i)}^{\oplus k(i)},\]
and
\[\Sigma''=(\sigma_{i}''\colon \Gamma \to S_{d_{l(i)}k(i)}).\]
We leave it as an exercise to verify that
\[\mdim_{\Sigma',p}(X,\rho)=\mdim_{\Sigma'',p}(X,\rho).\]
So it is enough to show that
\[\mdim_{\Sigma'',p}(X,\rho)\leq \mdim_{\Sigma,p}(X,\rho).\]
Let $M$ be the diameter of $(X,\rho).$ Fix $\varespilon>0.$   Let $F\subseteq \Gamma$ be finite, $\delta>0,$ and let $\delta''>0$ to be determined shortly. Suppose
\[\phi\in \Map(\rho,F,\delta'',\sigma_{i}''),\]
and for $1\leq s\leq k(i),$ define
\[\phi_{s}\colon \{1,\cdots,d_{l(i)}\}\to X\]
by
\[\phi_{s}(j)=\phi((s-1)d_{l(i)}+j).\]
For all $g\in F,$
\[\frac{1}{k(i)}|\{s:\rho_{2}(\phi_{s}\circ \sigma_{l(i)}(g),g\phi_{s})\geq \delta\}|\leq\frac{(\delta'')^{2}}{\delta^{2}}.\]
Thus
\[\frac{1}{k(i)}|\{s:\phi_{s}\in \Map(\rho,F,\delta,\sigma_{l(i)})\}|=\frac{1}{k(i)}\left|\bigcap_{g\in F}\{s:\rho_{2}(\phi_{s}\circ \sigma_{l(i)}(g),g\phi_{s})< \delta\}\right|\geq 1-|F|\frac{(\delta'')^{2}}{\delta^{2}}.\]
Choose $\delta''$ so that
\[\frac{(\delta'')^{2}}{\delta^{2}}M^{p}|F|<\varespilon^{p}.\]

	If $x_{0}\in X^{d_{l(i)}}$ is an arbitrary element, then
\[\Map(\rho,F,\delta'',\sigma_{i}'')\subseteq_{\varespilon,\rho_{p}}\bigcup_{\substack{A\subseteq\{1,\cdots,k(i)\},\\ |A|\geq k(i)\left(1-\frac{(\delta'')^{2}}{\delta^{2}}|F|\right)}}\{\phi \in X^{d_{l(i)}k(i)}:\phi_{a}\in \Map(\rho,F,\delta,\sigma_{l(i)})\mbox{ for $a\in A$},\phi_{a}=x_{0},a\notin A\}.\]
So, if $\frac{(\delta'')^{2}}{\delta^{2}}$ is sufficiently small,
\[S_{(2015)!\varepsilon}(\Map(\rho,F,\delta'',\sigma_{i}''),\rho_{p})\leq S_{\varespilon}(\Map(\rho,F,\delta,\sigma_{l(i)}),\rho_{p})^{k(i)}\cdot k(i)\binom{k(i)}{\lceil{\frac{(\delta'')^{2}}{\delta^{2}}|F|k(i)\rceil}}.\]
Applying Stirling's Formula as in Proposition \ref{P:finitepack}  and letting $\delta''\to 0,$ we see that
\[\mdim_{\Sigma'',p}((2015)!\varespilon,\rho)\leq \mdim_{\Sigma,p}(F,\delta,\varepsilon,\rho).\]
Now taking the infimum over $F,\delta$ completes the proof.

\end{proof}

By the above proposition, if $\Gamma$ is an amenable group and $X$ is a compact metrizable space with $\Gamma\actson X$ by homeomorphisms we may define
\[\mdim_{p}(X,\rho)=\mdim_{\Sigma,p}(X,\rho)\]
for any sofic approximation. We wish to show that we can compute
\[\mdim_{p}(X,\rho)\]
by taking limits of sizes of orbits over F\o lner sequences, in the manner which is typical of entropy for amenable groups. For this, we state a quasi-tiling Lemma due to Ornstein-Weiss.

\begin{lemma}[\cite{OrnWeiss}, page 24, Theorem 6]\label{L:quasitiling} Let $\Gamma$ be a countable discrete amenable group with F\o lner sequence $(F_{n})_{n=1}^{\infty}$ so that $e\in F_{n}$ for all $n.$  Let $\eta>0,$ then there is a $K\subseteq \Gamma$ finite,  a $\beta>0$ and integers $n_{1},\cdots,n_{k}$ so that if $F\subseteq \Gamma$ finite satisfies
\[\sup_{g\in K}|gF\Delta F|\leq \beta|F|,\]
then there are finite $c_{i,j}\in F,1\leq i\leq r(j),1\leq j\leq k$ and $F_{i,j}\subseteq F_{n_{j}}$ so that
\[|F_{i,j}|\geq (1-\eta)|F_{n_{j}}|\]
\[F_{i,j}c_{i,j}\subseteq F\]
\[\{F_{i,j}c_{i,j}\}\mbox{ are disjoint }\]
\[\left|\bigcup F_{i,j}c_{i,j}\right|\geq (1-\eta)|F|.\]
\end{lemma}

We also define canonical sofic approximations of an infinite amenable group. For $F\subseteq \Gamma$ finite, and for $g\in \Gamma,$ let $\tau_{F}(g)\colon F\setminus  g^{-1}F\to F\setminus gF$ be an arbitrary bijection. Define $\sigma_{F}\colon \Gamma\to \Sym(F)$ by $\sigma_{F}(g)(x)=gx$ for $x\in F\cap g^{-1}F,$ and $\sigma_{F}(g)(x)=\tau_{F}(g)(x)$ for $x\in F\setminus g^{-1}F.$ If $(F_{n})_{n=1}^{\infty}$ is a F\o lner sequence of $\Gamma,$ then $\sigma_{F_{n}}$ is a sofic approximation.

\begin{theorem} Let $\Gamma$ be a countable discrete infinite amenable group, and $X$ a compact metrizable space with $\Gamma\actson X$ by homeomorphisms. Let $\rho$ be a dynamically generating continuous pseudometric. Let $1\leq p<\infty,$ and $(F_{n})_{n=1}^{\infty}$ a F\o lner sequence for $\Gamma.$ Define $\rho_{F_{n},p}$ on $X$ by
\[\rho_{F_{n},p}(x,y)^{p}=\frac{1}{|F_{n}|}\sum_{g\in F_{n}}\rho(gx,gy)^{p}.\]
Then,
\[\mdim_{p}(X,\rho)=\liminf_{\varespilon\to 0}\limsup_{n\to \infty}\frac{\log S_{\varespilon}(X,\rho_{F_{n},p})}{|F_{n}|\log (1/\varespilon)}.\]
In particular, the right-hand side is independent of the choice of F\o lner sequence.
\end{theorem}

\begin{proof}

	Since we already know that $\mdim_{p}(X,\Gamma)$ is independent of the choice of sofic approximation, we use the sofic approximation $(\sigma_{F_{n}}\colon \Gamma\to \Sym(F_{n}))$ coming from $F_{n}$. For $x\in X,$ define
\[\phi_{x}\colon F_{n}\to X\]
by
\[\phi_{x}(g)=gx.\]
Using the definition of $\sigma_{F_{n}}$ it is not hard to show that for every $E\subseteq \Gamma$ finite, $\delta>0$ we have
\[\phi_{x}\in \Map(\rho,E,\delta,\sigma_{F_{n}})\]
for all large $n.$ As
\[\rho_{p}(\phi_{x},\phi_{y})=\rho_{F_{n},p}(x,y)\]
it follows that
\[\liminf_{\varespilon\to 0}\limsup_{n\to \infty}\frac{\log S_{\varespilon}(X,\rho_{F_{n},p})}{|F_{n}|\log (1/\varespilon)}\leq \mdim_{p}(X,\Gamma).\]

	For the reverse inequality, we may assume that $e\in F_{j}$ for all $j.$ Let $M$ be the diameter of $(X,\rho).$ We need the quasi-tiling Lemma mentioned before. Let $\varepsilon>0,$ set
\[\alpha=\limsup_{m\to \infty}\frac{\log S_{\varespilon}(X,\rho_{F_{m},p})}{|F_{m}|\log (1/\varespilon)}.\]
Let $\eta,\kappa>0$  be small  in a manner depending upon $\varepsilon$ to be determined later. We may assume that $\eta,\kappa<1.$ Let $N$ be such that for all $n\geq N$ we have
\[\frac{\log S_{\varespilon}(X,\rho_{F_{n},p})}{|F_{n}|\log (1/\varespilon)}\leq \kappa+\alpha.\]
 and let $n_{1},\cdots,n_{k}\in \NN$ with $n_{j}\geq N$ for all $j,$ $\beta>0$  and $ K\subseteq\Gamma$ finite be as Lemma \ref{L:quasitiling} for $\eta,(F_{n})_{n\in \NN}.$ Set
\[E=K\cup \bigcup_{j=1}^{k}F_{n_{j}},\]
and let $\delta>0$ be sufficiently small in a manner to be determined later.

	For $\phi \in \Map(\rho,E,\delta,\sigma_{F_{n}}),k\in E$ and for all large $n$ we have from the definition of $\sigma_{F_{n}}$ that
\[\frac{1}{|F_{n}|}\sum_{g\in F_{n}\cap k^{-1}F_{n}}\rho(\phi(kg),k\phi(g))^{2}\leq 2\delta^{2}.\]
Set
\[F=\bigcap_{k\in E}\{g\in F_{n}\cap k^{-1}F_{n}:\rho(\phi(kg),k\phi(g))<\varespilon\}.\]
Since $F_{n}$ is F\o lner for all large $n$ we have
\[|F|\geq \left(1-\frac{4\delta^{2}|E|}{\varespilon^{2}}\right)|F_{n}|.\]
We have for all $g\in\Gamma:$
\begin{align*}
\frac{|gF\Delta F|}{|F|}&\leq 2 \frac{|F_{n}\setminus F|}{|F|}+\frac{|gF_{n}\Delta F_{n}|}{|F|}\\
&\leq \frac{8\delta^{2}|E|}{\varepsilon^{2}}\frac{|F_{n}|}{|F|}+\frac{|gF_{n}\Delta F_{n}|}{|F_{n}|}\frac{|F_{n}|}{|F|}\\
&\leq \frac{8\delta^{2}|E|}{\varepsilon^{2}}\frac{1}{1-\frac{4\delta^{2}|E|}{\varepsilon^{2}}}+\frac{|gF_{n}\Delta F_{n}|}{|F_{n}|}\frac{1}{1-\frac{4\delta^{2}|E|}{\varespilon^{2}}}.
\end{align*}
Thus we can choose $\delta$  small enough so that for all large $n,$
\[\sup_{g\in K}|gF\Delta F|<\beta|F|.\]
We will want to take $\delta$ even smaller later.
Let $c_{i,j},F_{i,j},r(j)$ be as in Lemma \ref{L:quasitiling} for $F.$ Set $x_{i,j}=\phi(c_{i,j}).$ Then,
\[\sup_{h\in F_{i,j}}\rho(hx_{i,j},\phi(hc_{i,j}))<\varespilon.\]
If we choose $\eta<\frac{\varespilon^{p}}{M^{p}},$ we may then choose $\delta>0$ sufficiently small so that
\[M^{p}\left(1-(1-\eta)\left(1-\frac{4\delta^{2}|E|}{\varepsilon^{2}}\right)\right)<\varepsilon^{p}.\]

	Let $\Omega_{j}\subseteq X$ be a $\varepsilon$-dense subset with respect to $\rho_{F_{n_{j}},p}$ of minimal cardinality. Let $a\in X^{F_{n}\setminus \bigcup_{i,j}F_{i,j}c_{i,j}}$ be arbitrary. Given $y\in \prod_{i,j}\Omega_{j},$ let
\[\phi_{y}\colon F_{n}\to X\]
by given by
\[\phi_{y}(x)=\begin{cases}
gy(i,j),&\textnormal{ if $x\in F_{i,j}c_{i,j},x=gc_{i,j}$}\\
a(x),&\textnormal{ if $x\in F_{n}\setminus \bigcup_{i,j}F_{i,j}c_{i,j}$}
\end{cases}.\]
Now for $\phi$ as before, let $y_{i,j}$ be such that $\rho_{F_{n_{j}},p}(y_{i,j},x_{i,j})<\varepsilon,$ and let $y\in \prod_{i,j}\Omega_{j}$ be given by $y(i,j)=y_{i,j}.$Then,
\begin{align*}
\rho_{p}(\phi,\phi_{y})&\leq \left(M^{p}\left(1-(1-\eta)\left(1-\frac{4\delta^{2}|E|}{\varepsilon^{2}}\right)\right)+\frac{1}{|F_{n}|}\sum_{i,j}\sum_{g\in F_{i,j}}\rho(\phi(gc_{i,j}),gy_{i,j})^{p}\right)^{1/p}\\
&\leq \varepsilon+\left(\frac{1}{|F_{n}|}\sum_{i,j}\sum_{g\in F_{i,j}}\left(\varepsilon+\rho(gx_{i,j},gy_{i,j})\right)^{p}\right)^{1/p}\\
&\leq 2\varepsilon+\left(\frac{1}{|F_{n}|}\sum_{i,j}\sum_{g\in F_{i,j}}\rho(gx_{i,j},gy_{i,j})^{p}\right)^{1/p},
\end{align*}
the last two inequalities following by the triangle inequality for $\ell^{p}$-spaces. By choice of $y_{i,j}$ we have
\begin{align*}
\left(\frac{1}{|F_{n}|}\sum_{i,j}\sum_{g\in F_{i,j}}\rho(gx_{i,j},gy_{i,j})^{p}\right)^{1/p}&\leq \varepsilon\left(\frac{1}{|F_{n}|}\sum_{i,j}|F_{n_{j}}|\right)^{1/p}\\
&\leq \frac{\varepsilon}{(1-\eta)^{1/p}}\left(\frac{1}{|F_{n}|}\sum_{i,j}|F_{ij}|\right)^{1/p}\\
&=\frac{\varepsilon}{(1-\eta)^{1/p}}\left(\frac{1}{|F_{n}|}\sum_{i,j}|F_{i,j}c_{i,j}|\right)^{1/p}\\
&\leq \frac{\varepsilon}{(1-\eta)^{1/p}}\left(\frac{|F|}{|F_{n}|}\right)^{1/p}\\
&\leq\frac{\varepsilon}{(1-\eta)^{1/p}}
\end{align*}
we now assume that $(1-\eta)\geq 1/2.$
We then have
\[\rho_{p}(\phi,\phi_{y})\leq \varepsilon(2+2^{1/p}).\]

	Thus for all sufficiently small $\delta>0,$
\[S_{(2+2^{1/p})(2015)!\varespilon }(\Map(\rho,E,\delta,\sigma_{F_{n}}),\rho_{p})\leq \prod_{j=1}^{k}S_{\varespilon}(X,\rho_{F_{n_{j}},p})^{r(j)}.\]
So for all large $n,$
\[\frac{\log S_{(2+2^{1/p})\cdot(2015)!\varespilon}(\Map(\rho,E,\delta,\sigma_{F_{n}}),\rho_{p})}{|F_{n}|\log (1/\varepsilon)}\leq(\alpha+\kappa)\sum_{j=1}^{k}\frac{r(j)|F_{n_{j}}|}{|F_{n}|}.\]
We have
\[\sum_{j=1}^{k}\frac{r(j)|F_{n_{j}}|}{|F_{n}|}\leq \frac{1}{(1-\eta)|F_{n}|}\sum_{i,j}|F_{i,j}c_{i,j}|\leq \frac{1}{(1-\eta)}.\]
Taking infimum over $\delta>0,$ and then an infimum over the finite subsets of $\Gamma,$
\[\frac{\log(1/\varepsilon)-\log((2+2^{1/p})(2015)!)}{\log(1/\varepsilon)}\mdim_{\Sigma,p}\left((2+2^{1/p})\cdot(2015)!\varespilon,\rho\right)\leq (\alpha+\kappa)\frac{1}{1-\eta}.\]
As $\eta,\kappa$ can be made arbitrary small, we find that
\[\frac{\log(1/\varepsilon)-\log((2+2^{1/p})(2015)!)}{\log(1/\varepsilon)}\mdim_{\Sigma,p}\left((2+2^{1/p})\cdot (2015)!\varespilon,\rho\right)\leq\limsup_{n\to \infty}\frac{\log S_{\varespilon}(X,\rho_{F_{n},p})}{|F_{n}|\log (1/\varespilon)},\]
and the proof is completed by letting $\varepsilon\to 0.$

\end{proof}

Metric mean dimension for actions of an amenable group $\Gamma$ as defined by Lindenstrauss and Weiss in \cite{LindWeiss} is given by
\[\mdim(X,\rho)=\liminf_{\varespilon\to 0}\limsup_{n\to \infty}\frac{\log S_{\varespilon}(X,\rho_{F_{n},\infty})}{|F_{n}|\log (1/\varespilon)},\]
where $(F_{n})_{n=1}^{\infty}$ is a F\o lner sequence for $\Gamma.$ Thus this Theorem may be regarded as an $\ell^{p}$ analogue  of Theorem 5.1 in \cite{Li}. Note that as in \cite{Li} Theorem 3.1, we have to restrict to infinite groups. It is remarked in \cite{Li} Remark 3.8 that sofic mean dimension does not agree with the usual mean dimension when the group acting is finite. We essentially face the same difficulty as in \cite{Li}. However,  we can actually say what our $\ell^{p}$-version of metric mean dimension is when the group acting is finite. Again we focus on the case $1\leq p<\infty,$ as the case $p=\infty$ is handled in \cite{Li} Theorem 5.1.

For the next proposition, we will need to avoid our usual procedure of dropping the dependence on $k$ when we use $\ell^{p}$-product metrics. Thus if $\Gamma$ is a finite group, if $X$ is a compact metrizable space with $\Gamma\actson X$ by homeomorphisms, if $\rho$ is a dynamically generating pseudometric on $X$ and $1\leq p<\infty,$ we shall use $\rho_{\Gamma,p,k}$ for the pseudometric on $X^{k}$ given by
\[\rho_{\Gamma,p,k}(x,y)^{p}=\frac{1}{|\Gamma|k}\sum_{g\in\Gamma}\sum_{j=1}^{k}\rho(gx(j),gy(j))^{p},\]
for a finite set $A$ we also use $\rho_{p,A}$ for the pseudometric on $X^{A}$ given by
\[\rho_{p,A}(x,y)^{p}=\frac{1}{|A|}\sum_{j\in A}\rho(x(j),y(j))^{p}.\]
 \begin{proposition}\label{P:fintiegroupssuck} Let $\Gamma$ be a finite group, and $X$ a compact metrizable space with $\Gamma\actson X$ by homeomorphisms. Let $\rho$ be a dynamically generating continuous pseudometric on $X$ and $1\leq p<\infty.$ Then

 (i):
 \[\lim_{k\to\infty}\frac{\log S_{\varepsilon}(X^{k},\rho_{\Gamma,p,k})}{k}\]
 exists and is
 \[\inf_{k}\frac{\log S_{\varepsilon}(X^{k},\rho_{\Gamma,p,k})}{k}.\]

(ii): We have
\[\mdim_{p}(X,\rho)=\frac{1}{|\Gamma|}\liminf_{\varepsilon\to 0}\lim_{k\to\infty}\frac{\log S_{\varepsilon}(X^{k},\rho_{\Gamma,p,k})}{k\log(1/\varepsilon)}.\]

 \end{proposition}

\begin{proof}

(i): Set
\[a_{k}=\log S_{\varepsilon}(X^{k},\rho_{\Gamma,p,k}).\]
It suffices to show that $a_{k+l}\leq a_{k}+a_{l}.$ The desired result then follows from a well-known real analysis exercise. Fix $k,l\in \NN.$ Let $A_{1}\subseteq X^{k}$ be an $\varepsilon$-dense subset with respect to $\rho_{\Gamma,p,k}$ of minimal cardinality. Let $A_{2}\subseteq X^{l}$ be an $\varepsilon$-dense subset with respect to $\rho_{\Gamma,p,l}$ of minimal cardinality. For $x\in X^{k+l},$ let $x_{1}\in X^{k}$ be $x\big|_{\{1,\dots,k\}}$ and let $x_{2}\in X^{l}$ be defined by $x_{2}(j)=x(j+k).$ Let $S$ be the set of all $x\in X^{k+l}$ with $x_{j}\in A_{j}$ for $j=1,2.$ Given $y\in X^{k+l},$ choose $a\in A_{1},b\in A_{2}$ so that
\[\rho_{\Gamma,p,k}(y_{1},a)<\varepsilon\]
\[\rho_{\Gamma,p,l}(y_{2},b)<\varepsilon.\]
Let $x\in S$ be such that $x_{1}=a,x_{2}=b.$ Then,
\[\rho_{\Gamma,p,k+l}(x,y)<\varepsilon.\]
This shows that
\[S_{\varepsilon}(X^{k+l},\rho_{\Gamma,p,k+l})\leq S_{\varepsilon}(X^{k},\rho_{\Gamma,p,k})S_{\varepsilon}(X^{l},\rho_{\Gamma,p,l}).\]
Taking $\log$ of both sides we see that
\[a_{k+l}\leq a_{k}+a_{l}.\]

(ii): By Proposition \ref{P:independence} we can use any sofic approximation to compute $\mdim_{p}(X,\rho).$ Let
\[\sigma_{k}\colon \Gamma\to \Sym(\Gamma\times \{1,\dots,k\})\]
be given by
\[\sigma_{k}(g)(h,j)=(gh,j).\]
Then $\Sigma=(\sigma_{k})_{k=1}^{\infty}$ is a sofic approximation, and we will use this sofic approximation to do our calculation.

	Given $x\in X^{k},$ we let $\phi_{x}\colon \Gamma\times \{1,\dots,k\}\to X$ be defined by
\[\phi_{x}(g,j)=gx(j).\]
Then for all $\delta>0,$ and finite $F\subseteq \Gamma$ we have $\phi_{x}\in \Map(\rho,F,\delta,\sigma_{k}).$ Further
\[\rho_{p,\Gamma\times\{1,\dots,k\}}(\phi_{x},\phi_{y})=\rho_{\Gamma,p,k}(x,y).\]
Thus
\[S_{\varepsilon}(\Map(\rho,F,\delta,\sigma_{k}),\rho_{p,\Gamma\times\{1,\dots,k\}})\geq S_{2\varepsilon}(X^{k},\rho_{\Gamma,p,k}).\]
Dividing by $\log(1/2\varepsilon)|\Gamma|k,$ and letting $k\to\infty,$ then taking the infimum over $F,\delta$ and letting $\varepsilon\to 0$ we see that
\[\mdim_{\Sigma,p}(X,\rho)\geq \frac{1}{|\Gamma|}\liminf_{\varepsilon\to 0}\lim_{k\to\infty}\frac{\log S_{\varepsilon}(X^{k},\rho_{\Gamma,p,k})}{k\log(1/\varepsilon)}.\]

	We now return to the reverse inequality. Let $M$ be the diameter of $(X,\rho).$ Let $\varepsilon>0.$ Fix a $\delta>0.$ Given $\psi\in \Map(\rho,\Gamma,\delta,\sigma_{k})$ we let $x_{\psi}\in X^{k}$ be defined by
\[x_{\psi}(j)=\psi(e,j).\]
Then
\begin{align*}
\rho_{p,\Gamma\times\{1,\dots,k\}}(\phi_{x_{\psi}},\psi)^{p}&=\frac{1}{|\Gamma|k}\sum_{g\in\Gamma}\sum_{j=1}^{k}\rho(\psi(\sigma_{k}(g)(e,j)),g\phi(e,j))^{p}\\
&\leq \sum_{g\in\Gamma}\rho_{p,\Gamma\times\{1,\dots,k\}}(\psi\circ \sigma_{k}(g),g\psi)^{p}.
\end{align*}
If $p\leq 2,$ then by H\"{o}lder's inequality.
\[\rho_{p,\Gamma\times\{1,\dots,k\}}(\psi\circ \sigma_{k}(g),g\psi)\leq \rho_{2,\Gamma\times\{1,\dots,k\}}(\psi\circ\sigma_{k}(g),g\psi)<\delta,\]
Whereas if $p\geq 2,$ then
\[\rho_{p,\Gamma\times\{1,\dots,k\}}(\psi\circ\sigma_{k}(g),g\psi)^{p}\leq M^{p-2}\rho_{2,\Gamma\times\{1,\dots,k\}}(\psi\circ \sigma_{k}(g),g\psi)^{2}\leq M^{p-2}\delta^{2}.\]
In either case, we see that if $\delta$ is sufficiently small then
\[\rho_{p,\Gamma\times\{1,\dots,k\}}(\phi_{x_{\psi}},\psi)<\varepsilon\]
for all $\psi\in\Map(\rho,\Gamma,\delta,\sigma_{k}).$ Thus
\[\log S_{(2015)!\varepsilon}(\Map(\rho,\Gamma,\delta,\sigma_{k}),\rho_{p,\Gamma\times\{1,\dots,k\}})\leq \log S_{\varepsilon}(X^{k},\rho_{\Gamma,p,k}).\]
Dividing by $\log\left(\frac{1}{(2015)!\varepsilon}\right)|\Gamma|k,$ then letting $k\to\infty,$ then $\delta\to 0$ and then $\varepsilon\to 0$ we get the desired inequality.

\end{proof}

\section{Microstates Rank}\label{S:micr}

	Here we develop the notion of microstates rank as a middle road to the equality of von Neumann dimension and mean dimension. This is similar to \cite{LiLiang}, where the authors develop ``mean rank" as a middle road to proving the equality of  von Neumann-L\"{u}ck rank and mean dimension for amenable groups. Most of the ideas in this section are simply borrowed from \cite{Me}. Here, instead of using ``$\varepsilon$-dimension" we will use $\varepsilon$-covering numbers. However, the idea is essentially the same: we will recover von Neumann dimension by measuring the size of a microstates space of almost equivariant maps.

\begin{definition}\emph{ Let $\Gamma$ be a countable discrete sofic group, with sofic approximation $\sigma_{i}\colon \Gamma\to S_{d_{i}}.$ Let $A$ be a finitely generated $\ZZ(\Gamma)$-module, and $B$ a submodule of $A.$  Let $S=(a_{j})_{j=1}^{n}$ be a sequence in $A,$ such that $\Span\{ga_{j}:g\in \Gamma,1\leq j\leq n \}=A,$ and $T=(b_{j})_{j=1}^{\infty}$ a sequence in $B$ such that $\Span\{gb_{j}:g\in \Gamma,j\in \NN\}=B.$   For $F\subseteq \Gamma$ finite, $m\in \NN,\delta>0,$ we define $\Hom_{\Gamma}(S|T,F,m,\delta,\sigma_{i})$ to be the set of all abelian group homomorphisms $\Phi\colon A\to \ell^{2}_{\RR}(d_{i},\mu_{d_{i}})$ such that for all $1\leq j\leq n,1\leq k\leq m,g_{1},\cdots,g_{k}\in F$ we have}
\[\|\Phi(g_{1}\cdots g_{k}a_{j})-\sigma_{i}(g_{1})\cdots \sigma_{i}(g_{k})\Phi(a_{j})\|_{2}<\delta,\]
\[\|\Phi(a_{j})\|_{2}\leq 1,\mbox{ if $1\leq j\leq n$},\]
\[\|\Phi(b_{j})\|_{2}<\delta,\mbox{if $1\leq j\leq m$}.\]
\end{definition}

	If $\Gamma,S,T,A,B$ are as above, let $\rho_{A}$ be the pseudometric on abelian group homomorphisms defined by
\[\rho_{A}(\Phi,\Psi)^{2}=\sum_{j=1}^{n}\|\Phi(a_{j})-\Psi(a_{j})\|_{2}^{2}.\]

We now define the \emph{microstates rank} of $S$ given $T$ by
\[\micr(S|T)=\liminf_{\varepsilon\to 0}\inf_{F,m,\delta}\limsup_{i\to \infty}\frac{\log S_{\varepsilon}(\Hom_{\Gamma}(S|T,F,m,\delta,\sigma_{i}),\rho_{A})}{d_{i}\log(1/\varepsilon)}.\]
For getting the appropriate lower bound, we will need an approach closer to mean dimension (as opposed to metric mean dimension). For this,  we define the \emph{topological microstates rank} as follows. Set
\[\Xi(S|T,F,m,\delta,\sigma_{i})=\{(\Phi(a_{j}))_{j=1}^{n}:\Phi\in \Hom_{\Gamma}(S|T,F,m,\delta,\sigma_{i})\}.\]
Define

\[\tmicr(S|T)=\sup_{\varepsilon>0}\inf_{F,m,\delta}\limsup_{i\to \infty}\frac{1}{d_{i}}\wdim_{\varepsilon}([0,1/2]^{nd_{i}}\cap \Xi(S|T,F,m,\delta,\sigma_{i}),\|\cdot\|_{\infty}).\]
A priori, this definition depends on our choice on $S,T$ while we would really like a definition which only depends on the inclusion $B\subseteq A.$ However, we are really interested in the case $A=\ZZ(\Gamma)^{\oplus n},$ and for this we will be able to show that the microstates ranks is the von Neumann--L\"{u}ck rank of $A/B.$

	Let us recall and introduce some terminology. If $\Phi$ is a bounded operator on a Hilbert space, we use $|\Phi|=(\Phi^{*}\Phi)^{1/2}.$ If $\Phi$ is a normal matrix, and $A\subseteq \CC,$ we use
\[\chi_{A}(\Phi)=\sum_{\lambda\in A}\proj_{\ker(\Phi-\lambda I)},\]
where $\proj_{\ker(\Phi-\lambda I)}$ is the orthogonal projection onto $\ker(\Phi-\lambda I).$ Note that the above sum is finite. If $P,Q$ are orthogonal projections on a Hilbert space, then $P\wedge Q$ is the orthogonal projection onto $\im(P)\cap \im(Q).$ Additionally, $P\vee Q$ is the orthogonal projection onto $\overline{\im(P)+\im(Q)}.$ Lastly, we recall that if $A\in M_{n}(\CC),$ then $\tr(A)=\frac{1}{n}\Tr(A),$ and $\|A\|_{2}=(\tr(A^{*}A))^{1/2}.$
If $a\in \ZZ(\Gamma),j\in\{1,\dots,n\}$ we use $a\otimes e_{j}$ for the element in $\ZZ(\Gamma)^{\oplus n}$ which is $a$ in the $j^{th}$ coordinate and zero elsewhere.

\begin{proposition}\label{P:microstatesrank} Let $\Gamma$ be a sofic group, and $\Sigma$ a sofic approximation. Let $B\subseteq \ZZ(\Gamma)^{\oplus n}$ be a $\ZZ(\Gamma)$-submodule. Let $S=(e\otimes e_{1},\cdots,e\otimes e_{n}),$ and $T=(b_{j})_{j=1}^{\infty}$ a generating sequence for $B.$ Then
\[\tmicr(S|T)=\micr(S|T)=\vr(\ZZ(\Gamma)^{\oplus n}/B).\]
\end{proposition}

\begin{proof} Let $\Sigma=(\sigma_{i}\colon \Gamma\to S_{d_{i}}).$ Extend $\sigma_{i}$ to a map $\CC(\Gamma)\to M_{d_{i}}(\CC)$ by
\[\sigma_{i}(f)=\sum_{g\in \Gamma}\widehat{f}(g)\sigma_{i}(g),\]
and to a map $M_{n,m}(\CC(\Gamma))\to M_{n,m}(M_{d_{i}}(\CC))$ by
\[\sigma_{i}(f)_{jk}=\sigma_{i}(f_{jk}).\]
We shall prove that
\[\tmicr(S|T)\leq\micr(S|T),\]
\[\micr(S|T)\leq \vr(\ZZ(\Gamma)^{\oplus n}/B),\]
\[\vr(\ZZ(\Gamma)^{\oplus n}/B)\leq \tmicr(S|T).\]

	Define a continuous pseudometric $\Theta$ on $\TT^{n}$ by
\[\Theta(x,y)=\|x-y\|_{2},\]\
where the $\ell^{2}$-norm is with respect to the uniform probability measure.
By Lemma \ref{L:subsetsofproducts} we have that
\begin{align*}
\tmicr(S|T)&\leq \liminf_{\varepsilon\to 0}\inf_{F,m,\delta} \frac{1}{d_{i}}\limsup_{n\to\infty}\frac{\log S_{\varepsilon}(\Xi(S|T,F,m,\delta,\sigma_{i}),\Theta_{1})}{d_{i}\log(1/\varepsilon)}\\
&\leq  \liminf_{\varepsilon\to 0}\inf_{F,m,\delta} \frac{1}{d_{i}}\limsup_{n\to\infty}\frac{\log S_{\varepsilon}(\Xi(S|T,F,m,\delta,\sigma_{i}),\Theta_{2})}{d_{i}\log(1/\varepsilon)}
\end{align*}
 the last line following since $\|\cdot\|_{\ell^{1}(d_{i},\mu_{d_{i}})}\leq\|\cdot\|_{\ell^{2}(d_{i},\mu_{d_{i}})}.$ It is easy to see that
 \[\liminf_{\varepsilon\to 0}\inf_{F,m,\delta} \frac{1}{d_{i}}\limsup_{n\to\infty}\frac{\log S_{\varepsilon}(\Xi(S|T,F,m,\delta,\sigma_{i}),\Theta_{2})}{d_{i}\log(1/\varepsilon)}=\micr(S|T).\]
 So this proves one inequality.

Let us prove that $\tmicr(S|T)\geq \vr(\ZZ(\Gamma)^{\oplus n}/B).$  We will wish to regard $L(\Gamma)^{\oplus n},\ell^{2}(\Gamma)^{\oplus n},\ell^{2}(d_{i},u_{d_{i}})^{\oplus n}$ as column vectors. However, we will have to use elements of $\RR(\Gamma)^{\oplus n}$ to induce operators
\[\ell^{2}(\Gamma)^{\oplus n}\to \ell^{2}(\Gamma).\]
For this, recall that if we are given
\[\alpha=\begin{bmatrix}
\alpha_{1}\\
\vdots\\
\alpha_{n}
\end{bmatrix}\in L(\Gamma)^{\oplus n}\]
we set
\[\widetidle{\alpha}=\begin{bmatrix}
\alpha_{1}& \alpha_{2}&\cdots &\alpha_{n}
\end{bmatrix}.\]

	Given $T\in M_{d_{i}}(\RR)$ we define $1\otimes T\in M_{n}(M_{d_{i}}(\RR))$ by
\[(1\otimes T)_{ij}=\begin{cases}
T,& \textnormal{ if $i=j$}\\
0,& \textnormal{ if $i\ne j$.}
\end{cases}.\]
For a finite subset $F$ of $\Gamma$ containing the identity, $m\in \NN,\delta>0$ let $P_{F,m}(\delta,i)\in M_{n}(M_{d_{i}}(\RR))$ be the orthogonal projection
\[\chi_{[0,\delta)}\left(\sum_{j=1}^{m}\sigma_{i}(\widetilde{b_{j}})^{*}\sigma_{i}(\widetilde{b_{j}})\right)\wedge \bigwedge_{\substack{g_{1},\cdots,g_{k+1}\in F,\\1\leq k\leq m}}\chi_{[0,\delta)}(1\otimes |\sigma_{i}(g_{1})\cdots \sigma_{i}(g_{k+1})-\sigma_{i}(g_{1}\cdots g_{k+1})|^{2})\]
and
\[W_{F,m}(\delta,i)=P_{F,m}(\delta,i)\ell_{\RR}^{2}(d_{i},\mu_{d_{i}})^{\oplus n}.\]
For $\xi\in [0,1/2\sqrt{n}]^{nd_{i}}\cap W_{F,m}(\delta,i),$ define $\Phi_{\xi}\colon \ZZ(\Gamma)^{\oplus n}\to \ell^{2}_{\RR}(d_{i},\mu_{d_{i}})$ by
\[\Phi_{\xi}(f)=\sigma_{i}(\widetilde{f})\xi,\]
again regarding $L(\Gamma)^{\oplus n}\subseteq M_{1,n}(L(\Gamma)).$

	Then for $k\leq m,$
\[\|\Phi_{\xi}(b_{k})\|_{2}^{2}=\ip{\sigma_{i}(\widetilde{b_{k}})^{*}\sigma_{i}(\widetilde{b_{k}})\xi,\xi}=\ip{|\sigma_{i}(\widetilde{b_{k}})|^{2}\xi,\xi}<\delta .\]
Similarly, for $1\leq j\leq n,1\leq k\leq m$ and $g_{1},\dots ,g_{k}\in F,$
\begin{align*}
\|\sigma_{i}(g_{1})\cdots\sigma_{i}(g_{k}) \Phi_{\xi}(e\otimes e_{j})-\Phi_{\xi}(g_{1}\cdots g_{k}\otimes e_{j})\|_{2}^{2}&=\|\sigma_{i}(g_{1})\cdots \sigma_{i}(g_{k})\sigma_{i}(e)(\xi(j))-\sigma_{i}(g_{1}\cdots g_{k})\xi(j)\|_{2}^{2}\\
&\leq \|(1\otimes (\sigma_{i}(g_{1})\cdots \sigma_{i}(g_{k})\sigma_{i}(e)-\sigma_{i}(g_{1}\cdots g_{k})))\xi\|_{2}^{2}\\
&=\ip{(1\otimes |(\sigma_{i}(g_{1})\cdots \sigma_{i}(g_{k})\sigma_{i}(e)-\sigma_{i}(g_{1}\cdots g_{k}))|^{2})\xi,\xi} \\
 &<\delta .
\end{align*}
Thus,
\begin{align*}
\wdim_{\varepsilon}(\Xi(S|T,F,m,\sqrt{\delta},\sigma_{i})\cap [0,1/2]^{nd_{i}},\|\cdot\|_{\infty})&\geq \wdim_{\varepsilon}\left((1\otimes\sigma_{i}(e))\left(W_{F,m}(\delta,i)\cap\left[0,\frac{1}{2\sqrt{n}}\right]^{nd_{i}}\right),\|\cdot\|_{\infty}\right)\\
&=\wdim_{\varepsilon}\left(W_{F,m}(\delta,i)\cap \left[0,\frac{1}{2\sqrt{n}}\right]^{nd_{i}},\|\cdot\|_{\infty}\right),
\end{align*}
as $\sigma_{i}(e)$ is a permutation.

By \cite{Tsuka} Appendix A (see additionally \cite{Gromov} page 11, and \cite{Gourn1} Lemma 2.4) we have
\[\wdim_{\varepsilon}\left(W_{F,m}(\delta,i)\cap \left[0,\frac{1}{2\sqrt{n}}\right]^{nd_{i}},\|\cdot\|_{\infty}\right)\geq \dim(W_{F,m}(\delta,i))\]
if $\varepsilon<\frac{1}{4\sqrt{n}}.$
By soficity,
\[\|\sigma_{i}(g_{1})\cdots\sigma_{i}(g_{k})\sigma_{i}(e)-\sigma_{i}(g_{1}\cdots g_{k})\|_{2}\to 0,\]
and
\[\tr\left(\chi_{(\delta,\infty)}(|\sigma_{i}(g_{1})\cdots\sigma_{i}(g_{k})\sigma_{i}(e)-\sigma_{i}(g_{1}\cdots g_{k})|^{2}\right))\leq\frac{1}{\delta}\|\sigma_{i}(g_{1})\cdots\sigma_{i}(g_{k})\sigma_{i}(e)-\sigma_{i}(g_{1}\cdots g_{k})\|_{2}^{2}.\]
Thus
\[\limsup_{i\to \infty}\frac{1}{d_{i}}\dim(W_{F,m}(\delta,i))=\limsup_{i\to \infty}\tr\otimes \Tr\left(\chi_{[0,\delta)}\left(\sum_{j=1}^{m}\sigma_{i}(\widetilde{b_{j}})^{*}\sigma_{i}(\widetilde{b_{j}})\right)\right).\]
We have
\[\left\|\sum_{j=1}^{m}\sigma_{i}((\widetilde{b_{j}})^{*}b_{j})-\sum_{j=1}^{m}\sigma_{i}(\widetilde{b_{j}})^{*}\sigma_{i}(\widetilde{b_{j}})\right\|_{2}\to 0.\]
Applying Lemma $\ref{L:weak^{*}}$, we see that the spectral measure with respect to $\tr$ of
\[\sum_{j=1}^{m}\sigma_{i}(\widetilde{b_{j}})^{*}\sigma_{i}(\widetilde{b_{j}})\]
converges weak$^{*}$ to the spectral measure with respect to the group trace on $L(\Gamma)$ of
\[\sum_{j=1}^{m}(\widetilde{b_{j}})^{*}\widetilde{b_{j}}.\]
Hence, we may argue as in the von Neumann dimension Lemma (Lemma \ref{L:vnDLemma})  to show that
\begin{align*}
\inf_{m,\delta}\limsup_{i\to \infty}\tr\otimes \Tr\left(\chi_{[0,\delta)}\left(\sum_{j=1}^{m}\sigma_{i}(\widetilde{b_{j}})^{*}\sigma_{i}(\widetilde{b_{j}})\right)\right)&\geq \inf_{m}\dim_{R(\Gamma)}\left(\ker \sum_{j=1}^{m}\rho(\widetilde{b_{j}})^{*}\rho(\widetilde{b_{j}})\right)\\
&=n-\sup_{m}\dim_{R(\Gamma)}\left(\overline{\im \sum_{j=1}^{n}\rho(\widetilde{b_{j}})^{*}\rho(\widetilde{b_{j}})}\right).
\end{align*}
For $\alpha\in \CC(\Gamma),$  and $1\leq j\leq k$ we have
\[\rho(\widetilde{b_{j}})^{*}\alpha=\begin{bmatrix}
\alpha b_{j}(1)\\
\alpha b_{j}(2)\\
\vdots\\
\alpha b_{j}(k)
\end{bmatrix}\]
with
\[b_{j}=\begin{bmatrix}
b_{j}(1)\\
b_{j}(2)\\
\vdots\\
b_{j}(k)
\end{bmatrix}.\]
It thus easily follows that
\[\overline{\im \sum_{j=1}^{n}\rho(\widetilde{b_{j}})^{*}\rho(\widetilde{b_{j}})}\subseteq \mathcal{H}_{B}.\]
where $\mathcal{H}_{B}$ is defined as in Lemma \ref{L:hilred}. This implies that
\[\tmicr(S|T)\geq \vr(\ZZ(\Gamma)^{\oplus n}/B),\]
by Lemma \ref{L:hilred}.

	For the inequality
	\[\micr(S|T)\leq \vr(\ZZ(\Gamma)^{\oplus n}/B),\]
 fix $\eta>0,k\in \NN.$ Let $f\in M_{k,n}(\ZZ(\Gamma))$ be defined by
	\[f=\begin{bmatrix}
\widetilde{b_{1}}\\
\widetilde{b_{2}}\\
\vdots\\
\widetilde{b_{k}}
\end{bmatrix}.\]
 Note that if $F,m$ are big enough, and $\delta>0$ is small enough then for all $\Phi\in \Hom_{\Gamma}(S|T,F,m,\delta,\sigma_{i})$ we have
\[\left\|\sigma_{i}(f)\begin{bmatrix}
\Phi(e\otimes e_{1})\\
\Phi(e\otimes e_{2})\\
\vdots\\
\Phi(e\otimes e_{n})
\end{bmatrix}\right\|_{2}<\eta.\]
Therefore
\begin{align*}\left\|\chi_{[0,\sqrt{\eta}]}(|\sigma_{i}(f)|)\begin{bmatrix}
\Phi(e\otimes e_{1})\\
\Phi(e\otimes e_{2})\\
\vdots\\
\Phi(e\otimes e_{n})
\end{bmatrix}-\begin{bmatrix}
\Phi(e\otimes e_{1})\\
\Phi(e\otimes e_{2})\\
\vdots\\
\Phi(e\otimes e_{n})
\end{bmatrix}\right\|_{2}&=
\left\|\chi_{(\sqrt{\eta},\infty)}(|\sigma_{i}(f)|)\begin{bmatrix}
\Phi(e\otimes e_{1})\\
\Phi(e\otimes e_{2})\\
\vdots\\
\Phi(e\otimes e_{n})\end{bmatrix}\right\|_{2}\\
&\leq \frac{1}{\sqrt{\eta}}\left\|\sigma_{i}(f)\begin{bmatrix}
\Phi(e\otimes e_{1})\\
\Phi(e\otimes e_{2})\\
\vdots\\
\Phi(e\otimes e_{n})\end{bmatrix}\right\|_{2}\\
&<\sqrt{\eta}.\end{align*}
It  follows that
\begin{align*}
\micr(S|T)&\leq \liminf_{\varepsilon\to 0}  \inf_{\eta}\limsup_{i\to \infty}\frac{\log S_{\varespilon}\left(\chi_{[0,\eta]}(|\sigma_{i}(f)|)\sqrt{n}\Ball(\ell^{2}_{\RR}(d_{i},\mu_{d_{i}})^{\oplus n}),\|\cdot\|_{\ell^{2}(d_{i},\mu_{d_{i}})^{\oplus n}}\right)}{d_{i}\log(1/\varepsilon)}\\
&=\dim_{R(\Gamma)}(\ker \rho(f))=n-\dim_{R(\Gamma)}(\ker\rho(f)^{\perp}).
\end{align*}
where in the last line we use the von Neumann Dimension Lemma (Lemma \ref{L:vnDLemma}).

	Now,
\[\ker(\rho(f))^{\perp}=\overline{\im\rho(f^{*})}.\]
For $x\in L(\Gamma),$ define $1\otimes x\in M_{n}(L(\Gamma))$ by
\[(1\otimes x)_{ij}=\begin{cases}
x,& \textnormal{ if $i=j$,}\\
0, &\textnormal{ if $i\ne j.$}
\end{cases}\]
As in Lemma \ref{L:hilred}, we have
\[\overline{\im\rho(f^{*})}=\overline{\left\{\sum_{s=1}^{k}(1\otimes \alpha_{s})b_{s}:\alpha_{s}\in \CC(\Gamma), 1\leq s\leq k\right\}}.\]
It follows that as $k\to \infty,$
\[P_{\overline{\im\rho(f^{*})}}\to P_{\mathcal{H}_{B}}\]
in the strong operator topology (here $P_{\mathcal{H}}$ is the projection onto the subspace $\mathcal{H}$). Thus it remains to apply Lemma
\ref{L:hilred}.

\end{proof}

	The main application of topological microstates rank as defined above is to show that
\[\mdim_{\Sigma}(\widehat{A},\Gamma)\geq \vr(A).\]

	It is also an analogue of Peters' result in \cite{Peters} on computing entropy of an algebraic action in terms of the action on the dual module. Interestingly, there is no known analogue for computing entropy in the sofic case (however, see \cite{Me5}  Proposition 2.6 for a similar formula in a special case). It is also an analogue of mean rank as defined in \cite{LiLiang}.

\section{Relative Mean Dimension}\label{S:RMD}

	For our proof, it will be useful to define relative mean dimension for  a submodule $B\subseteq A.$ Thinking of mean dimension as analogous to entropy, relative mean dimension should be analogous to entropy of $\Gamma \actson \widehat{A},$ given $\Gamma\actson \widehat{B}.$ It will essentially be defined by taking the microstates for $\Gamma\actson \widehat{A}$ and asking that as elements of $\widehat{A}^{d_{i}},$ they are small on $B.$ This is analogous to the author's approach to the proof of invariance in \cite{Me} (see for example Definition 2.11 in \cite{Me}).  We  will show that the relative mean dimension of $B\subseteq A$  is the mean dimension of $\widehat{A/B}.$

	The main case of interest for relative mean dimension is $A=\ZZ(\Gamma)^{\oplus n},$ where it is easy to construct microstates $\phi\colon \{1,\cdots,d_{i}\}\to \widehat{A}=(\TT^{n})^{\Gamma}.$ For example, take any $\xi\in (\TT^{d_{i}})^{n},$ and define $\phi(j)(g)(l)=\xi(l)(\sigma_{i}(g^{-1})(j)),$ and this will be a microstate for enough values of $\xi$ to prove that $\mdim_{\Sigma}(\widehat{A},\Gamma)=n.$  However, it is quite hard (potentially impossible) to force the microstates $\phi\colon \{1,\cdots,d_{i}\}\to \widehat{A}$ defined above to take values in $\widehat{A/B}.$ It turns out to be quite simple to force $\phi$ to be small on $B.$  For example, let $\Phi\in \Hom_{\Gamma}(S|T,F,m,\delta,\sigma_{i})$ as in the preceding section and define $\xi(j)=\Phi(e_{j})+\ZZ^{d_{i}},$ and this will work. By the preceding section, there are \emph{many} such $\Phi,$ and this will turn out to give us enough elements to get us the lower bound on metric mean dimension.

Recall, that if $\|\cdot\|$ is a norm on $\RR^{n},$ then we will also use $\|\cdot\|$ for the dynamically generating continuous pseudometric on $(\TT^{n})^{\Gamma},$ so that the distance between $f$ and $g$ is
\[\|f(e)-g(e)+\ZZ^{n}\|,\]
in particular
\[|x+\ZZ|=\inf_{k\in \ZZ}|x+k|.\]

\begin{definition} \emph{Let $\Gamma$ be a sofic group with sofic approximation $\Sigma.$ Let $B\subseteq A$ be $\ZZ(\Gamma)$-modules. Fix a dynamically generating continuous pseudometric $\rho$ on $\widehat{A},$ and $T=(b_{j})_{j=1}^{\infty}$ a sequence in $B$ so that $\{gb_{j}:g\in \Gamma,j\in \NN\}$ generates $B$ as an abelian group. For $F\subseteq \Gamma$ finite, $m\in \NN,\delta>0,$ let $\Map(\rho|T,F,m,\delta,\sigma_{i})$ to be the set of all $\phi\in \Map(\rho,F,\delta,\sigma_{i})$ such that}
\[\sup_{1\leq k\leq m}\frac{1}{d_{i}}\sum_{j=1}^{d_{i}}|\phi(j)(b_{k})|^{2}<\delta^{2}.\]
\end{definition}
With notation as above, we define the $p$-mean dimension of $A$ given $T$ with respect to $\rho$ by
\[\mdim_{\Sigma,p}(\rho|T)=\liminf_{\varepsilon\to 0}\inf_{F,m,\delta}\limsup_{i\to \infty}\frac{\log S_{\varepsilon}(\Map(\rho|T,F,m,\delta,\sigma_{i}),\rho_{p})}{d_{i}\log(1/\varespilon)}.\]

	Finally, define the $p$-metric mean dimension of $A$ given $T$ by
\[\mdim_{\Sigma,p}(A|T,\Gamma)=\inf_{\rho}\mdim_{\Sigma,p}(\rho|T),\]
where the infimum is over all dynamically generating continuous pseudometrics. We will show that for all$ 1\leq p<\infty,$ and for any dynamically generating continuous pseudometric $\rho$ we have,
\[\mdim_{\Sigma,p}(\rho|T)=\mdim_{\Sigma,p}(\widehat{A/B},\rho\big|_{\widehat{A/B}}).\]

\begin{lemma}\label{L:smallquo} Let $B\subseteq A$ be discrete abelian groups, and let $\rho$ be a continuous pseudometric on $\widehat{A}.$  Let $S \subseteq B$ be a  subset such that $\left\{\sum_{j=1}^{n}a_{j}:a_{j}\in S\right\}=B.$ Then for all $\varepsilon>0$, there is an $E\subseteq S$ finite, and a $\delta>0$ so that whenever $\chi\in\widehat{A},$ and $|\chi(b)|<\delta$ for all $b\in E,$  there is a $\widetilde{\chi}\in \widehat{A/B}$ with
\[\rho(\widetidle{\chi},\chi)<\varepsilon,\]
(considering $\widehat{A/B}\subseteq \widehat{A}$ in the usual way).

\end{lemma}

\begin{proof} This is proved by contradiction, using that $\widehat{A}$ is compact and will be left as an exercise to the reader.

\end{proof}

\begin{proposition}\label{P:relativemeandimension} Let $\Gamma$ be a sofic group with sofic approximation $\Sigma.$ Let $B\subseteq A$ be countable $\ZZ(\Gamma)$-modules. Let $T=(b_{j})_{j=1}^{\infty}$  be a sequence in $B$ such that $\{gb_{j}:j\in \NN,g\in \Gamma\}$ generates $B$ as an abelian group. For any dynamically generating continuous pseudometric $\rho$ on $\widehat{A}$ and $1\leq p<\infty,$ we have
\[\mdim_{\Sigma,p}(\rho|T)=\mdim_{\Sigma,p}(\widehat{A/B},\rho\big|_{\widehat{A/B}}).\]

\end{proposition}

\begin{proof} Throughout the proof, we shall view $\widehat{A/B}\subseteq \widehat{A}$ in the usual way. Let  $M$ be the diameter of $(\widehat{A},\rho).$  Fix $e\in F=F^{-1}\subseteq  \Gamma$ finite, $\delta>0.$ Let $\eta>0$ in a manner to be determined later. By the preceding Lemma, there is an $m\in \NN,\delta_{0}>0,$ and $E\subseteq \Gamma$  finite so that if $\chi\in \widehat{A},$ and $|\chi(hb_{k})|<\delta_{0}$ for all $k\leq m,$ and $h\in E,$  then there is a $\widetidle{\chi}\in \widehat{A/B}$ so that
\[\sup_{g\in F}\rho(g\widetilde{\chi},g\chi)<\eta.\]

	Let $d$ be the dynamically generating continuous pseudometric on $\widehat{A}$ defined by
\[d(\chi_{1},\chi_{2})=\rho(\chi_{1},\chi_{2})+\left(\sum_{j=1}^{m}|\chi_{1}(b_{j})-\chi_{2}(b_{j})|^{2}\right)^{1/2}.\]
Fix $F',\delta'>0$ which will depend upon $\eta,F,m,E,\delta_{0}$ in a manner to be determined. Specifically, we will apply Lemma \ref{L:heart} and assume that
\[\Map(\rho,F',\delta',\sigma_{i})\subseteq \Map\left(d,E\cup \{e\}\cup E^{-1},\frac{\delta_{0}^{2015}}{(2015)!|E|m},\sigma_{i}\right),\]
and that $\delta'<\delta,F'\supseteq F.$ Let $\phi\in \Map(\rho|T,F',m,\delta',\sigma_{i}).$ We may find a  $W_{i}\subseteq \{1,\cdots,d_{i}\}$ of size at least $\left(1-4m(|E|+1)\left(\frac{\delta'}{\delta_{0}}\right)^{2}\right)d_{i}$  so that
\[|\phi(j)(b_{k})|<\frac{\delta_{0}}{2},\]
\[|\phi(j)(hb_{k})-\phi(\sigma_{i}(h^{-1})j)(b_{k})|<\frac{\delta_{0}}{2},\]
for all $j\in W_{i},k\leq m,h\in E.$

  Set $V_{i}=W_{i}\cap \bigcap_{h\in E}\sigma_{i}(h^{-1})^{-1}(W_{i}),$ then
\[|V_{i}^{c}|\leq d_{i}4m(|E|+1)^{2}\left(\frac{\delta'}{\delta_{0}}\right)^{2},\]
and
\[|\phi(j)(hb_{k})|<\delta_{0}\]
for all $j\in V_{i},h\in E,1\leq k\leq m.$ If $j\in V_{i},$ we can find a $\phi^{0,j}\in \widehat{A/B}$ so that
\[\rho(g\phi^{0,j},g\phi(j))<\eta\]
for all $g\in F.$ Define $\phi^{0}\colon \{1,\cdots,d_{i}\}\to \widehat{A/B}$ by
\[\phi^{0}(j)=\phi^{0,j}\]
if $j\in V_{i},$ and $\phi^{0}(j)=0$ otherwise. For all $g\in F,$ we have
\[\rho_{p}(g\phi^{0},g\phi)^{p}\leq \eta^{p}+M^{p}(|E|+1)^{2}4m\left(\frac{\delta'}{\delta_{0}}\right)^{2}.\]
Now choose $\eta$ to be any number less than $\delta.$ This forces $m,\delta_{0}$ and $E$ on us, but we then choose $\delta'$ to be sufficiently small so that
\[\max(\rho_{2}(g\phi^{0},g\phi),\rho_{p}(g\phi^{0},g\phi))<\max((\delta^{p}+\delta M^{p})^{1/p},(\delta^{2}+\delta M^{2})^{1/2}),\]
for all $g\in F.$ Set
\[\kappa(\delta)=\max((\delta^{p}+\delta M^{p})^{1/p},(\delta^{2}+\delta M^{2})^{1/2}).\]
Since $e\in F,$ we know that $\phi^{0}\in \Map(\rho\big|_{\widehat{A/B}},F,2\kappa(\delta)+\delta,\sigma_{i}),$ and thus
\[\Map(\rho|T,F',m,\delta',\sigma_{i})\subseteq_{\kappa(\delta),\rho_{p}}\Map(\rho\big|_{\widehat{A/B}},F,2\kappa(\delta)+\delta,\sigma_{i}).\]

	Since
\[\lim_{\delta\to 0}\kappa(\delta)=0\]
we conclude that
\[\mdim_{\Sigma,p}(\rho|T)\leq \mdim_{\Sigma,p}(\widehat{A/B},\rho\big|_{\widehat{A/B}}).\]
By our convention that $\widehat{A/B}\subseteq \widehat{A},$ it follows that
\[\Map(\rho\big|_{\widehat{A/B}},F,\delta,\sigma_{i})\subseteq \Map(\rho|T,F,m,\delta,\sigma_{i})\] for all $m,$ so the reverse inequality is easier.

\end{proof}

To apply to metric mean dimension we will need a Lemma on extending metrics, for which we make the following definition.

\begin{definition} \emph{Let $(X,d)$ be a metric space of finite diameter, and denoted by $\BC(X)$ the space of bounded continuous functions on $X.$ The} Kuratowski embedding \emph{is given as the map $x\mapsto d_{x}$ from $X$ to $\BC(X),$ where $d_{x}(y)=d(x,y).$ It is a standard exercise that this is an isometric embedding.}\end{definition}

\begin{lemma}\label{L:metricextension} Let $Y\subseteq X$ be compact metrizable spaces. Let $\rho$ be a compatible metric on $Y.$ Then there is a compatible metric $d$ on $X$ such that $d\big|_{Y\times Y}=\rho.$

\end{lemma}

\begin{proof} The basic idea of the proof is that a metric on a compact space $K$ is equivalent (via the Kuratowski embedding) to a continuous embedding $K\to V,$ with $V$ a Banach space. So we simply find a Banach space $V$ containing $C(Y),$ and an embedding $X\to V$ which extends the Kuratowski embedding on $Y.$

By the Tietze Extension Theorem, we may extend $\rho$ to a continuous function $f\colon X\times Y\to [0,\infty).$ Define a map
\[\Phi\colon X\to C(Y),\]
by $\Phi(x)(y)=f(x,y).$ Note that $\Phi$ extends the Kuratowski embedding of $Y$ into $C(Y)$, but may not be injective. To fix injectivity, we take a direct sum.

	Fix a compatible metric $\Delta$ on $X,$ and define
\[\alpha\colon X\to C(X)\]
by
\[\alpha(x)(y)=\Delta(x,\{y\}\cup Y).\]
Define
\[\Psi\colon X\to C(Y)\oplus C(X),\]
by
\[\Psi(x)=(\Phi(x),\alpha(x)).\]
Note that $\Psi$ is continuous. We claim that $\Psi$ is injective. Suppose that $a,b\in X$ and $\Psi(a)=\Psi(b).$ First suppose $a\in Y.$ Then $\alpha(b)(a)=\alpha(a)(a)=0,$ which implies that $b\in \{a\}\cup Y=Y.$ Since $\Phi$ extends the Kuratowski embedding and $\Phi(a)=\Phi(b),$ we conclude that $a=b.$ If $a\notin Y,$  then $\alpha(a)(b)=\alpha(b)(b)=0,$ so $a\in \{b\}\cup Y,$ so $a=b,$ since $a$ is not in $Y.$ Thus $\Psi$ is injective. Once we know that $\Psi$ is injective, the lemma is proved by defining
\[d(x,y)=\max(\|\Phi(x)-\Phi(y)\|_{\infty},\|\alpha(x)-\alpha(y)\|_{\infty}).\]
 Since $\Phi$ extends the Kuratowski embedding, it is not hard to see that $d$ extends $\rho.$

\end{proof}

\begin{cor}\label{C:relativereduce} Let $\Gamma$ be a countable discrete sofic group, and $\Sigma$ a sofic approximation. Let $B\subseteq \ZZ(\Gamma)^{\oplus n}$ be a $\ZZ(\Gamma)$-submodule. Then for any sequence $T$ in $B$ generating $B,$ and for all $1\leq p<\infty,$
\[\mdim_{\Sigma,M,p}((\ZZ(\Gamma)^{\oplus n}/B)^{\widehat{}},\Gamma)=\inf_{\rho}\mdim_{\Sigma,p}(\rho|T),\]
where the infimum is over all compatible metrics $\rho$ on $(\TT^{n})^{\Gamma}.$
\end{cor}

\begin{proof} Use the preceding lemma and Proposition \ref{P:relativemeandimension}.\end{proof}

\section{Proof of The Main Theorem}\label{S:upperbound}

	We now prove the Main Theorem.

\begin{theorem}\label{T:main} Let $\Gamma$ be a countable discrete sofic group, and $A$ a finitely generated $\ZZ(\Gamma)$-module. Then for all $1\leq p\leq \infty,$
\[\mdim_{\Sigma,M,p}(\widehat{A},\Gamma)=\vr(A).\]
Further if $a_{1},\cdots,a_{n}$ generate $A,$ and
\[\rho(\chi_{1},\chi_{2})=\left(\sum_{j=1}^{n}|\chi_{1}(a_{j})-\chi_{2}(a_{j})|^{2}\right)^{1/2},\]
then
\[\mdim_{\Sigma,M}(\widehat{A},\rho)=\vr(A).\]
\end{theorem}

\begin{proof} Let $B$ be the kernel of the $\ZZ(\Gamma)$-linear map
\[\ZZ(\Gamma)^{\oplus n}\to A\]
sending $e_{j}$ to $a_{j}.$  Let $T=(b_{j})_{j=1}^{\infty}$ be a sequence in $B$ so that $\Gamma T$ generates $B$ as an abelian group. Let $\Delta$ be the dynamically generating pseduometric on $(\TT^{n})^{\Gamma}$ given by
\[\Delta(\chi_{1},\chi_{2})=\|\chi_{1}(e)-\chi_{2}(e)\|_{\ell^{2}(n)}.\]
Using  Corollaries \ref{C:relativereduce},\ref{C:mdimcomparison} and Propositions \ref{P:finitepack},\ref{P:relativemeandimension}  it suffices to show that
\[\mdim_{\Sigma,2}(\Delta|T)\leq \vr(A),\]
and that for any compatible metric $d$ on $(\TT^{n})^{\Gamma},$
\[\mdim_{\Sigma,1}(d|T)\geq \vr(A).\]
We shall first prove the upper bound.

Fix $k\in \NN,$ and $\varepsilon>0,$ let $\kappa>0$ depending upon $\varespilon$ in a manner to be determined. Let $F\subseteq \Gamma$ be finite, $m\in \NN$ and $\delta>0,$ which will depend upon $k,\kappa$ in a manner to be determined later. Let $q\colon (\RR^{d_{i}})^{n}\to ((\RR/\ZZ)^{d_{i}})^{n}$ be the quotient map. Let $\phi\in \Map(\Delta|T,F,m,\delta,\sigma_{i}),$ define $\zeta_{\phi}\in (\TT^{d_{i}})^{n}$ by
\[\zeta_{\phi}(a)(j)=\phi(j)(e)(a),\]
(recall that $\phi$ is a map $\{1,\cdots,d_{i}\}\to (\TT^{n})^{\Gamma}$). Let $\xi_{\phi}\in ([-1,1]^{d_{i}})^{n}$ be such that
\[q(\xi_{\phi})=\zeta_{\phi}.\]
 Extend $\sigma_{i}$ to $\ZZ(\Gamma)$ by
\[\sigma_{i}(f)=\sum_{g\in \Gamma}\widehat{f}(g)\sigma_{i}(g),\]
and then to a map $\sigma_{i}\colon M_{p,r}(\ZZ(\Gamma))\to M_{p,r}(M_{d_{i}}(\ZZ))$ by $\sigma_{i}(f)_{st}=\sigma_{i}(f_{st}).$  Let $f\in M_{k,n}(\ZZ(\Gamma))$ be given by
\[f=\begin{bmatrix}
\widetilde{b_{1}}\\
\widetilde{b_{2}}\\
\vdots\\
\widetilde{b_{k}}
\end{bmatrix}.\]
By soficity,
\[\frac{1}{d_{i}}\left|\bigcap_{\substack{g\in\supp(\widehat{b}_{s}),\\ 1\leq s\leq k}}\{1\leq j\leq d_{i}:\sigma_{i}(g^{-1})(j)=\sigma_{i}(g)^{-1}(j)\}\right|\to_{i\to\infty}1.\]
Thus if $F,m$ are sufficiently large, and $\delta>0$ is sufficiently small, then for all large $i$ we have
\[\|q(\sigma_{i}(f)\xi_{\phi})\|_{2}^{2}<\kappa^{2}.\]

	We may find an $l\in (\ZZ^{d_{i}})^{k}$ such that
\begin{equation}\label{E:almostsmall}
\|\sigma_{i}(f)\xi_{\phi}-l\|_{2}<\kappa.
\end{equation}
We may bound $\|\sigma_{i}(f)\xi_{\phi}\|_{\infty}$ in terms of $\|\widehat{b_{j}}\|_{\infty},|\supp(\widehat{b_{j}})|.$ Because of this, there is an $N$ (depending on k)  such that we may choose $l$ as above with $l\in (\{-N,\cdots,N\}^{d_{i}})^{k}.$  Let $\Omega$ be the set of $l \in (\{-N,\cdots,N\}^{d_{i}})^{k}$ so that there exists $\phi\in \Map(\Delta|T,F,m,\delta,\sigma_{i})$ with
\[\|\sigma_{i}(f)\xi_{\phi}-l||_{2}<\kappa.\]
For each $l\in \Omega$ choose $\phi_{l}$ such that
\[\|\sigma_{i}(f)\xi_{\phi_{l}}-l\|_{2}<\kappa.\]
If $\xi_{\phi}$ satisfies (\ref{E:almostsmall}), then
\[\|\sigma_{i}(f)(\xi_{\phi}-\xi_{\phi_{l}})\|_{2}<2\kappa.\]
By functional calculus,
\begin{align*}
\|\chi_{[0,\sqrt{2\kappa}]}(|\sigma_{i}(f)|)(\xi_{\phi}-\xi_{\phi_{l}})-(\xi_{\phi}-\xi_{\phi_{l}})\|_{2}^{2}&=\ip{\chi_{(\sqrt{2\kappa},\infty)}(|\sigma_{i}(f)|)(\xi_{\phi}-\xi_{\phi_{l}}),\xi_{\phi}-\xi_{\phi_{l}}}\\
&\leq \frac{1}{2\kappa}\ip{|\sigma_{i}(f)|^{2}(\xi_{\phi}-\xi_{\phi_{l}}),\xi_{\phi}-\xi_{\phi_{l}}}\\
&=\frac{1}{2\kappa}\|\sigma_{i}(f)(\xi_{\phi}-\xi_{\phi_{l}})\|_{2}^{2}\\
&<2\kappa.
\end{align*}

	Hence, the set of all such $\xi_{\phi}$ is contained in the $\sqrt{2\kappa}$-neighborhood with respect to $\|\cdot\|_{\ell^{2}(n,\ell^{2}_{\RR}(d_{i},\mu_{d_{i}}))}$ of
\[\bigcup_{l\in \Omega}\xi_{\phi_{l}}+(2\sqrt{n})\chi_{[0,\sqrt{2\kappa}]}(|\sigma_{i}(f)|)\Ball(\ell^{2}_{\RR}(d_{i},\mu_{d_{i}})^{\oplus n}).\]
From this,
\[S_{2(\varepsilon+\sqrt{2\kappa})}(\Map(\Delta|T,F,m,\delta,\sigma_{i}),\Delta_{2})\leq S_{\frac{\varepsilon}{2\sqrt{n}}}(\chi_{[0,\sqrt{2\kappa}]}(|\sigma_{i}(f)|)(\Ball(\ell^{2}_{\RR}(d_{i},\mu_{d_{i}})^{\oplus n})),\|\cdot\|_{\ell^{2}(n,\ell^{2}_{\RR}(d_{i},u_{d_{i}})})C(k)^{kd_{i}},\]
where $C(k)>0$ depends only on $k.$ Now choose $\sqrt{2\kappa}<\varespilon,$ then
\[S_{4\varepsilon}(\Map(\Delta |T,F,m,\delta,\sigma_{i}),\Delta_{2})\leq S_{\frac{\varepsilon}{2\sqrt{n}}}(\chi_{[0,\sqrt{2\kappa}]}(|\sigma_{i}(f)|)(\Ball(\ell^{2}_{\RR}(d_{i},\mu_{d_{i}})^{\oplus n})),\|\cdot\|_{\ell^{2}(n,\ell^{2}_{\RR}(d_{i},u_{d_{i}})})C(k)^{kd_{i}}.\]
As $F,m$ can be made arbitrarily large, and $\delta>0$ arbitrarily small, we can apply the von Neumann dimension Lemma (Lemma \ref{L:vnDLemma})  to see that
\begin{equation}\label{E:mmdub1}\mdim_{\Sigma,2}(\Delta|T)\leq \dim_{R(\Gamma)}(\ker\rho(f)).\end{equation}
Here we urge the reader to read the words of caution before Lemma \ref{L:hilred}. Now,
\begin{equation}\label{E:mmdub2}
\ker(\rho(f))^{\perp}=\overline{\im\rho(f^{*})}.\end{equation}

  Let
\[\eta\colon \ZZ(\Gamma)\to \ell^{2}(\Gamma)\]
be the injection given by
\[\eta(\alpha)(g)=\widehat{\alpha}(g),\]
and let
\[\lambda\colon \Gamma\to \mathcal{U}(\ell^{2}(\Gamma))\]
be the left regular representation, and  extend to $\CC(\Gamma)$ in the obvious way. Also let
\[\lambda\otimes 1\colon \CC(\Gamma)\to B(\ell^{2}(\Gamma)^{\oplus n})\]
be given by
\[(\lambda\otimes 1)(\alpha)\begin{bmatrix}
\xi_{1}\\
\xi_{2}\\
\vdots\\
\xi_{n}\end{bmatrix}
=\begin{bmatrix}
\lambda(\alpha)\xi_{1}\\
\lambda(\alpha)\xi_{2}\\
\vdots\\
\lambda(\alpha)\xi_{n}
\end{bmatrix}.\]
Then for $f_{1},\cdots,f_{k}\in \CC(\Gamma),$ we have
\[\rho(f^{*})\begin{bmatrix}
\eta(f_{1})\\
\eta(f_{2})\\
\vdots\\
\eta(f_{k})
\end{bmatrix}=\sum_{j=1}^{k}(\lambda\otimes 1)(f_{j})\eta^{\oplus n}(b_{j}).\]
Hence as $k\to \infty,$
\[\dim_{R(\Gamma)}(\overline{\im\rho(f^{*})})\to \dim_{R(\Gamma)}(\mathcal{H}_{B})\]
with the notation as in Lemma \ref{L:hilred}. Combining (\ref{E:mmdub1}),(\ref{E:mmdub2}) now proves the upper bound by Lemma \ref{L:hilred}.

 To prove the Lower bound, let $T,B$ be as before and set $S=(e\otimes e_{1},e\otimes e_{2},\dots,e\otimes e_{n}) $ where $e\otimes e_{j}$ is $e$ in the $j^{th}$ coordinate and zero elsewhere. Let $d$ be a compatible metric on $(\TT^{n})^{\Gamma}.$    Let $\Theta$ be the pseudometric on $(\TT^{n})^{\Gamma}$ given by
\[\Theta(x,y)=\|x(e)-y(e)\|_{\ell^{\infty}(n)}.\] By Lemma \ref{L:subsetsofproducts} and Lemma \ref{L:heart} we have
\begin{align*}
&\mdim_{\Sigma,1}(d|T)\geq \sup_{\varepsilon>0}\inf_{F,m,\delta}\limsup_{i\to \infty}\frac{\wdim_{\varepsilon}(\Map(\Delta|T,F,m,\delta,\sigma_{i}),\Theta_{\infty})}{d_{i}}.
\end{align*}
We may define a map
\[\Psi\colon \Hom(\ZZ(\Gamma)^{\oplus n},\ell^{2}_{\RR}(d_{i},\mu_{d_{i}}))\to \Hom(\ZZ(\Gamma)^{\oplus n},\TT^{d_{i}}),\]
by
\[\Psi(\Phi)=q\circ \Phi.\]
As
\[\Hom(\ZZ(\Gamma)^{\oplus n},\TT^{d_{i}})\cong ((\TT^{n})^{\Gamma})^{d_{i}}\]
canonically,  we may regard $\Psi$ as a map $\Hom(\ZZ(\Gamma)^{\oplus n},\ell^{2}_{\RR}(d_{i},\mu_{d_{i}}))\to ((\TT^{n})^{\Gamma})^{d_{i}}.$ Via this identification,
\[\Psi(\Hom_{\Gamma}(S|T,F\cup\{e\}\cup F^{-1},m,\delta,\sigma_{i}))\subseteq \Map(\Delta|T,F,m,n\delta,\sigma_{i})\]
when $i$ is large enough.
Since
\[q\big|_{[0,1/2]^{nd_{i}}}\colon [0,1/2]^{nd_{i}}\to \TT^{nd_{i}}\]
is $\|\cdot\|_{\infty}-\Theta_{\infty}$ isometric, the above discussion implies that
\begin{align*}
\wdim_{\varepsilon}(\Xi(S|T,F\cup\{e\}\cup F^{-1},m,\delta,\sigma_{i})\cap[0,1/2]^{nd_{i}},\|\cdot\|_{\infty})&\leq \wdim_{\varepsilon}(\Map(\Delta|T,F,m,n\delta,\sigma_{i}),\Theta_{\infty}).\\
\end{align*}
 Hence,
\[\mdim_{\Sigma,1}(d|T)\geq \tmicr(S|T)=\vr(A),\]
by Proposition \ref{P:microstatesrank}.

\end{proof}

\section{Mean Dimension of Finitely Presented Modules}

In this section we calculate mean dimension, instead of metric mean dimension, of algebraic actions $\Gamma\actson \widehat{A}$ when $A$ is a finitely presented $\ZZ(\Gamma)$-module and $\Gamma$ is residually finite. We leave it to the reader to verify that if $\Gamma$ is a residually finite group, and $\Gamma_{k}$ is a decreasing sequence of normal, finite-index subgroups with
\[\bigcap_{k=1}^{\infty}\Gamma_{k}=\{e\},\]
then the sequence of maps
\[(\sigma_{k}\colon\Gamma\to \Sym(\Gamma/\Gamma_{k}))\]
given by
\[\sigma_{k}(g)(x\Gamma_{k})=gx\Gamma_{k}\]
is a sofic approximation.
	
		We remark that there is some difficulty in extending this calculation to the general sofic case. Let $f\in M_{m,n}(\ZZ(\Gamma))$ where $\Gamma$ is a countable discrete sofic group with sofic approximation $\sigma_{i}\colon\Gamma\to S_{d_{i}}.$ Set
	\[A=\ZZ(\Gamma)^{\oplus n}/\ZZ(\Gamma)^{\oplus m}f,\]
and view $\widehat{A}\subseteq (\TT^{\Gamma})^{\oplus n}.$ For $\xi\in (\TT^{d_{i}})^{n}$ we can define
\[\phi_{\xi}\colon \{1,\dots,d_{i}\}\to (\TT^{\Gamma})^{\oplus n}\]
by
\[\phi_{\xi}(j)(l)(g)=\xi(l)(\sigma_{i}(g)^{-1}(j)),\]
and this will be a good enough microstate for the action $\Gamma\actson (\TT^{\Gamma})^{\oplus n}$ if $i$ is sufficiently large. Suppose now that $\xi\in \ker(\sigma_{i}(f)),$ as explained in Section \ref{S:RMD}, this will imply that the uniform probability measure of the set of $j$ for which $\phi_{\xi}(j)$ is ``close'' to $\widehat{A}$ tends to $1.$ Theorem \ref{T:main} essentially shows that this is enough to show the desired relation between metric mean dimension and von Neumann-L\"{u}ck rank. This is because inequality (\ref{E:canonicalperturbationinequality}) shows that $S_{\varepsilon}(E,d)$ is well-behaved under passing to small neighborhoods of a set $E.$ Thus for the purposes of \emph{metric} mean dimension, it is acceptable to consider microstates for $\Gamma\actson (\TT^{\Gamma})^{\oplus n}$ which are ``close'' to $\widehat{A}.$ However, this sort of trick is unacceptable if we wish to compute mean dimension.  For example, there is no general inequality of the form  (\ref{E:canonicalperturbationinequality}) for $\wdim_{\varepsilon}(E,d)$ and so we cannot use this perturbation argument. This discussion reveals that mean dimension is actually somewhat of a ``singular'' object, being quite sensitive to perturbations, an undesirable feature which is absent for metric mean dimension. This is somewhat curious, as it is may be possible that metric mean dimension and mean dimension always agree. See Question 6.6 in \cite{Li}. It has been shown by Lindenstrauss that if $X$ is a compact metrizable space, if $\ZZ\actson X$ by homeomorphisms, and the action $\ZZ\actson X$ has a nontrivial minimal factor, then the mean dimension and the metric mean dimension of $\ZZ\actson X$ are the same ( see \cite{Lin} Theorem 4.3). Potential ways to get around this issue are  discussed more in Section \ref{S:?}.

	For now, let us comment that the only way we are aware to avoid this issue is to force $\phi_{\xi}$ as defined above to be literally in $\widehat{A},$ as opposed to merely ``close'' to $\widehat{A}.$ This is significantly simpler when $\Gamma$ is residually finite as we now show.  The following calculation has been done in the amenable (without assuming residually finite) case by Li-Liang ( see \cite{LiLiang} Theorem 1.1).

\begin{theorem} Let $\Gamma$ be a countable discrete residually finite group, and let $\Gamma_{k}$ be a decreasing sequence of finite-index, normal subgroups such that $\bigcap_{k=1}^{\infty}\Gamma_{k}=\{e\}.$ Let $\sigma_{k}\colon \Gamma\to \Sym(\Gamma/\Gamma_{k})$ be given by
\[\sigma_{k}(g)(x\Gamma_{k})=gx\Gamma_{k}.\]
Set $\Sigma=(\sigma_{k})_{k=1}^{\infty}.$  For any finitely presented $\ZZ(\Gamma)$-module $A,$
\[\mdim_{\Sigma}(\widehat{A},\Gamma)=\vr(A).\]
\end{theorem}
\begin{proof} Since
\[\mdim_{\Sigma}(\widehat{A},\Gamma)\leq \mdim_{\Sigma,M}(\widehat{A},\Gamma)\]
by Theorem 6.1 of \cite{Li}, it suffice by Theorem \ref{T:main} to show that
\[\mdim_{\Sigma}(\widehat{A},\Gamma)\geq \vr(A).\]
Since $A$ is finitely presented, without loss of generality we may assume that
\[A=\ZZ(\Gamma)^{\oplus n}/\ZZ(\Gamma)^{\oplus m}f\]
for some $f\in M_{m,n}(\ZZ(\Gamma)).$  By the argument of  Lemma 5.4 in \cite{LiLiang} and the remarks preceding Lemma 5.3 of \cite{LiLiang} we see that
\[\vr(A)=\dim_{L(\Gamma)}(\ker\lambda(f)).\]
Let $\pi_{k}\colon\Gamma\to\Gamma/\Gamma_{k}$ be the quotient map. Extend $\pi_{k}$ to a map $\CC(\Gamma)\to \CC(\Gamma/\Gamma_{k})$ by linearity. We further extend $\pi_{k}$ to a map
\[M_{m,n}(\CC(\Gamma))\to M_{m,n}(\CC(\Gamma/\Gamma_{k}))\]
by
\[\pi_{k}(A)_{ij}=\pi_{k}(A_{ij})\mbox{ for $1\leq i\leq m,1\leq j\leq n.$}\]
Then
\[\pi_{k}(M_{m,n}(\ZZ(\Gamma)))\subseteq M_{m,n}(\ZZ(\Gamma/\Gamma_{k})).\]
We view $\CC(\Gamma/\Gamma_{k})\subseteq M_{\Gamma/\Gamma_{k}}(\CC)$ by the left regular representation. We leave it is an exercise to check that under this identification we have
\[\sigma_{k}(A)=\pi_{k}(A)\]
for all $A\in M_{m,n}(\CC(\Gamma)).$ We may then view
\[X_{f}\subseteq(\TT^{\Gamma})^{n}.\]
We will use the dynamically generating pseudometric  $\theta$ on $(\TT^{\Gamma})^{\oplus n}$ given by
\[\theta(\xi,\zeta)=\max_{1\leq l\leq n}|\xi(l)(e)-\zeta(l)(e)|.\]
We also use $\rho$ for the metric on $\RR^{s}$ given by
\[\rho(x,y)=\|x-y\|_{\infty}.\]
We will suppress notation and not include the dependence on $s.$ Let $q\colon \RR^{s}\to \TT^{s}$ be the canonical quotient map. Again this depends upon $s,$ but we will suppress notation as what value of $s$ we will be referring to will be clear from context.

	Since $\pi_{k}(f)\in M_{m,n}(\ZZ(\Gamma/\Gamma_{k}))$ we can view it acting on $(\RR^{\Gamma/\Gamma_{k}})^{n}$ as well as $(\TT^{\Gamma/\Gamma_{k}})^{n}.$ Let
\[\xi \in\ker(\pi_{k}(f))\cap \ell^{2}_{\RR}(\Gamma/\Gamma_{k},u_{\Gamma/\Gamma_{k}})^{\oplus n}.\]
Define
\[\phi_{\xi}\colon \Gamma/\Gamma_{k}\to (\TT^{\Gamma})^{\oplus n}\]
by
\[\phi_{\xi}(x)(l)(g)=\xi(l)(\pi_{k}(g)^{-1}x)+\ZZ.\]
A direct computation shows that for all $\alpha\in \ZZ(\Gamma)^{\oplus n}$ we have
\[\ip{\phi_{\xi}(x),\alpha}=(\pi_{k}(\alpha)\xi)(x)+\ZZ.\]
Thus for all $\beta\in\ZZ(\Gamma)^{\oplus m}$ we have
\[\ip{\phi_{\xi}(x),\beta f}=(\pi_{k}(\beta)\pi_{k}(f)\xi)(x)+\ZZ\]
so that $\phi_{\xi}(x)\in X_{f}.$ Furthermore, we have
\[\phi_{\xi}(\pi_{k}(g)(x))=g\phi_{\xi}(x).\]
So $\phi_{\xi}\in \Map(\theta,F,\delta,\sigma_{k})$ for all finite $F\subseteq\Gamma$ and $\delta>0.$ Since $q\big|_{([0,1/2]^{\Gamma/\Gamma_{k}})^{\oplus n}}$ is $\rho_{\infty}-\theta_{\infty}$ isometric, we find that
\[\wdim_{\varepsilon}(\Map(\theta,F,\delta,\pi_{k}),\theta_{\infty})\geq \wdim_{\varepsilon}(([0,1/2]^{\Gamma/\Gamma_{k}})^{\oplus n}\cap\ker(\pi_{k}(f)),\rho_{\infty}).\]
If $\varepsilon<1/4,$ then by Appendix A in \cite{Tsuka}, we know that
\[\wdim_{\varespilon}(([0,1/2]^{\Gamma/\Gamma_{k}})^{\oplus n}\cap \ker(\pi_{k}(f)),\rho_{\infty})\geq \dim_{\RR}(\ker(\pi_{k}(f)\cap (\RR^{\Gamma/\Gamma_{k}})^{\oplus n})).\]
Thus
\[\mdim_{\Sigma}(\widehat{A},\Gamma)\geq \lim_{k\to \infty}\frac{\dim_{\RR}(\ker\pi_{k}(f)\cap (\RR^{\Gamma/\Gamma_{k}})^{\oplus n})}{[\Gamma\colon\Gamma_{k}]}=\dim_{L(\Gamma)}(\ker \lambda(f)),\]
by \cite{Luck2} Theorem 2.3.2.

\end{proof}

\section{Failure of Addition Formula And Other Applications}

	It is known that if $\Gamma$ is amenable, and  we have a $\Gamma$-equivariant exact sequence of compact abelian groups
\[\begin{CD}
0 @>>> K_{1} @>>> K_{2} @>>> K_{3} @>>> 0
\end{CD}\]
then the entropy of the action on $K_{2}$ is the sum of the entropies of $K_{3}$ and $K_{1}.$ This was first proved by Yuzvinski\v\i\ in \cite{Yuz} for $\Gamma=\ZZ,$ and by Lind-Schmidt-Ward in \cite{LindSchmidt2} for $\Gamma=\ZZ^{d}.$ It was proved in the above generality by Li in \cite{Li2}. Additionally Li-Liang in \cite{LiLiang} prove an analogue for mean dimension for algebraic actions of amenable groups. Here we note that this fails for mean dimension of nonamenable groups.  Let $\FF_{n}$ denote the free group on $n$ letters. By Corollary 10.3.7 (iv) in \cite{Passman}, we can find a (left) $\ZZ(\FF_{n})$-submodule $M$ of $\ZZ(\FF_{n})$ with $M\cong \ZZ(\FF_{n})^{\oplus n}$ as a $\ZZ(\FF_{n})$-submodule. We then have an exact sequence of $\ZZ(\FF_{n})$-modules
\[\begin{CD}
0 @>>> \ZZ(\FF_{n})^{\oplus n} @>>> \ZZ(\FF_{n}) @>>> \ZZ(\FF_{n})/M @>>> 0
\end{CD}.\]
By Pontryagin duality, this induces a $\Gamma$-equivariant exact sequence of compact abelian groups
\[\begin{CD}
0 @>>>  (\ZZ(\FF_{n})/M)^{\widehat{}}  @>>> \TT^{\FF_{n}} @>>> (\TT^{\FF_{n}})^{\oplus n}@>>> 0
\end{CD}.\]
By Theorem \ref{T:main},
\[\mdim_{\Sigma,M}( \TT^{\FF_{n}},\FF_{n})=\vr(\ZZ(\FF_{n}))=1\]
\[\mdim_{\Sigma,M}((\TT^{\FF_{n}})^{\oplus n},\FF_{n})=\vr(\ZZ(\FF_{n})^{\oplus n})=n.\]
Because of this, we make the following definition.

\begin{definition}\emph{ Let $\Gamma$ be a countable discrete group with sofic approximation $\Sigma.$ We say that $(\Gamma,\Sigma)$ is} metric mean dimension additive for the class of finitely generated modules \emph{if for every exact sequence}
\[\begin{CD}
0 @>>> A @>>> B @>>> C @>>> 0
\end{CD}\]
\emph{of finitely generated $\ZZ(\Gamma)$-modules we have}
\[\mdim_{\Sigma,M}(\widehat{B},\Gamma)=\mdim_{\Sigma,M}(\widehat{A},\Gamma)+\mdim_{\Sigma,M}(\widehat{C},\Gamma).\]
\end{definition}

\begin{proposition} Let $\Gamma$ be a countable discrete sofic group with sofic approximation $\Sigma,$ suppose that $\Lambda$ is a subgroup of $\Gamma.$ If $(\Lambda,\Sigma\big|_{\Lambda})$ is not metric mean dimension additive for the class of finitely generated modules, then neither is $(\Gamma,\Sigma).$ In particular, if $\Gamma$ contains a nonabelian free subgroup, then $(\Gamma,\Sigma)$ is not metric mean dimension additive for the class of finitely generated modules.
\end{proposition}

\begin{proof}

	Let
\[\begin{CD}
0 @>>> A_{1} @>>> A_{2} @>>> A_{3} @>>> 0,
\end{CD}\]
be an exact sequence of finitely generated $\ZZ(\Lambda)$-modules with
\[\mdim_{\Sigma,M}(\widehat{A_{2}},\Lambda)\ne \mdim_{\Sigma,M}(\widehat{A_{1}},\Lambda)+\mdim_{\Sigma,M}(\widehat{A_{3}},\Lambda).\]
For $j=1,2,3$ set
\[A_{j}'=\ZZ(\Gamma)\otimes_{\ZZ(\Lambda)}A_{j}.\]
Then,
\[\L(\Gamma)\otimes_{\ZZ(\Gamma)}(\ZZ(\Gamma)\otimes_{\ZZ(\Lambda)}A_{j})\cong \L(\Gamma)\otimes_{\ZZ(\Lambda)}A_{j}\cong L(\Gamma)\otimes_{L(\Lambda)}(L(\Lambda)\otimes_{\ZZ(\Lambda)}A_{j}).\]
Hence by \cite{Luck}, Theorem 6.29 (2) we have
\[\vr(A_{j}')=\vr(A_{j}),\]
the proposition now follows from Theorem \ref{T:main}.

\end{proof}
As a consequence of Theorem \ref{T:main}, Conjecture 6.48 in \cite{Luck} predicts that every nonamenable sofic group is not metric mean dimension additive for the class of finitely generated modules (for any sofic approximation). I am grateful to Andreas Thom for pointing this out to me.

	The remaining applications are as in \cite{LiLiang}. For example, if one knows some version of the Atiyah conjecture for $\Gamma,$ then there are nice restrictions on the values of metric mean dimension. For this, one must note that if
\[A=\ZZ(\Gamma)^{\oplus n}/B\]
and we write $B$ as a union of finitely generated submodules $B_{k},$ then
\[\vr(A)=\lim_{k\to \infty}\vr(\ZZ(\Gamma)^{\oplus n}/B_{k}).\]
This was implicitly proved in Lemma \ref{L:hilred}. One can use this to show that if $\Gamma$ is a sofic group satisfying the strong Atiyah conjecture and with a bound on the order of finite subgroups, then for a finitely generated $\ZZ(\Gamma)$ module $A,$ the topological entropy of $\Gamma\actson\widehat{A}$ is finite if and only if the action has zero metric mean dimension. This requires extending previously known results on topological entropy from amenable groups to sofic groups, see  \cite{Me5} Theorem 5.3.
	
\section{Questions and Conjectures}\label{S:?}

		The main question is the following.

\begin{?}\emph{Is it possible to improve on the techniques of this proof to show in fact that $\mdim_{\Sigma}(\widehat{A},\Gamma)=\vr(A)$ for all sofic groups $\Gamma$ and finitely generated $\ZZ(\Gamma)$-modules $A?$} \end{?}

	Here is the main difficulty. Lemma $\ref{L:smallquo}$ allows us to show
\[N=\{\chi\in (\TT^{n})^{\Gamma}:|\chi(b_{j})|<\delta\mbox{ for all $1\leq j\leq k$}\}\subseteq_{\varepsilon,\rho}\widehat{A},\]
for any continuous pseudometric $\rho$ on $(\TT^{n})^{\Gamma},$ and with $\delta$ small enough depending on $\varepsilon,\rho.$
However, for the proof of this fact we  use a compactness argument in an essential way. Thus we do not produce a \emph{continuous} map
\[f\colon N\to \widehat{A},\]
so that $\rho(f(x),x)<\varespilon.$
	
	If we could find such a map, then one could solve the above question in the affirmative. For this, let us borrow some intuition from differential geometry. Let $M$ be a compact Riemannian manifold, and $N\subseteq M$ a submanifold. Then it is known that one can find a ``tubular neighborhood" for the inclusion $N\subseteq M.$ That is, if $\delta>0$ is sufficiently small, then there is a continuous map from the $\delta$-neighborhood of $N$ back to $N,$ which maps a point to a closest point in $N$ (this is of course a reflection of the exact Hilbert space structure of the tangent space). Now, one can try to think of $\widehat{A}$ as an infinite dimensional submanifold of $(\TT^{n})^{\Gamma}$ and one could hope that the same fact should hold.  It does not seem very clear how to make the intuition precise.

\begin{?} \emph{Can one remove the finite generation assumption in all of our results?}\end{?}

The notion of $p$-metric mean dimension develops naturally from our proof. By the nature of von Neumann dimension one is naturally led to use Hilbert space techniques, and this  naturally leads to $p$-metric mean dimension, at least for $p=2.$ A natural question is:
\begin{?}\emph{Let $\Gamma$ be a sofic group with sofic approximation $\Sigma,$ and $1\leq p<\infty.$ Is it true that for all compact metrizable spaces and $\Gamma\actson X,$ we have $\mdim_{\Sigma,M,p}(X,\Gamma)=\mdim_{\Sigma,M}(X,\Gamma)?$}\end{?}

	A negative answer to the above would show that mean dimension is not the same as metric mean dimension. On the other hand, if this question were true one could use $\rho_{\infty}$ and $\rho_{2}$ interchangeably. Since our microstates are defined by being $\rho_{2}$ almost equivariant, this could simplify many arguments, as it can be unnatural to use a uniform and probabilistic notion of closeness at the same time. Thus either answer to the above question would be worthwhile.

\end{document}